\documentclass[12pt,a4paper]{article}

\usepackage{amsmath}
\usepackage{amssymb}
\usepackage{mathtools}
\usepackage{tikz}
\usepackage{graphicx}
\usepackage{amsthm}
\usepackage{fullpage}
\usepackage{authblk}

\usepackage[backgroundcolor=white, textwidth=2cm]{todonotes}

\newcommand{\negphantom}[1]{\settowidth{\dimen0}{#1}\hspace*{-\dimen0}}

\theoremstyle{plain}
\newtheorem{thm}{Theorem}[section]
\newtheorem{cor}[thm]{Corollary}
\newtheorem{lem}[thm]{Lemma}

\theoremstyle{definition}
\newtheorem{dfn}[thm]{Definition}

\DeclareMathOperator{\aut}{Aut}
\DeclareMathOperator{\pow}{Pow}
\DeclareMathOperator{\soc}{Soc}
\DeclareMathOperator{\psl}{PSL}
\DeclareMathOperator{\psu}{PSU}

\DeclareMathOperator{\sz}{Sz}
\DeclareMathOperator{\spl}{SL}
\DeclareMathOperator{\alt}{Alt}
\DeclareMathOperator{\sym}{Sym}

\usepackage[colorlinks]{hyperref}
\definecolor{lightblue}{rgb}{0.5,0.5,1.0}
\definecolor{darkred}{rgb}{0.5,0,0}
\definecolor{darkgreen}{rgb}{0,0.5,0}
\definecolor{darkblue}{rgb}{0,0,0.5}

\hypersetup{colorlinks,linkcolor=darkred,filecolor=darkgreen,urlcolor=darkred,citecolor=darkblue}

\begin{document}
\title{
On groups with chordal power graph, including a classification in the case of finite simple groups
}

\author{Jendrik Brachter}
\author{Eda Kaja}
\affil{TU Darmstadt
}
\date{\today}
\maketitle

\begin{abstract}
We prove various properties on the structure of groups whose power graph is chordal. Nilpotent groups with this property have been classified by Manna, Cameron and Mehatari [The Electronic Journal of Combinatorics, 2021]. Here we classify the finite simple groups with chordal power graph, relative to typical number theoretic oracles. We do so by devising several sufficient conditions for the existence and non-existence of long cycles in power graphs of finite groups. 

We examine other natural group classes, including special linear, symmetric, generalized dihedral and quaternion groups, and we characterize direct products with chordal power graph. The classification problem is thereby reduced to directly indecomposable groups and we further obtain a list of possible socles. Lastly, we give a general bound on the length of an induced path in chordal power graphs, providing another potential road to advance the classification beyond simple groups.
\end{abstract}

\section{Introduction}

The study of graphs defined on the set of elements of a finite group currently is among the highly active areas in finite group theory. 
Recent contributions include work on the power graph and the enhanced power graph~\cite{bera_bhuniya,aalipour_akbari_cameron_nikandish_shaveisi_2017, abawajy_kelarev_chowdhury_2013, cameron_ghosh_2011},
the commuting graph~\cite{comm_graph_2022, morgan_parker_2013},
the non-generating graph~\cite{cameron_freedman_roney-dougal, lucchini_nemmi_2022}, the soluble graph~\cite{burness_lucchini_nemmi_2021},
and so on. The overarching goal of this line of research is to gain a better understanding of groups using graph theoretic techniques. In doing so, new combinatorial perspectives in group structure theory are developed. A comprehensive overview on graphs defined on groups can be found in~\cite{cameron2021graphs}.

In the present work we are concerned with the power graph of a group. Recall that the \emph{power graph} of a group $G$ is the simple graph whose vertex set is the set of elements of $G$ and an edge between two elements $x$ and $y$ exists if and only if one of $x=y^m$ or $y=x^m$ holds for some $m\in\mathbb{N}$. The concept goes back to an analogous definition for semigroups~\cite{ChakrabartyGhoshSen}.

As with many graphs defined on groups, early works on the power graph often investigate the expressiveness of the resulting graphs by considering widely-used graph theoretic parameters. These include connectedness~\cite{pandakrishna}, completeness or planarity~\cite{ChakrabartyGhoshSen}, among others. More generally the aim is to gain a detailed understanding of the subgraph structure. 

From the perspective of graph structure theory, an important tool in this regard is the framework of forbidden induced subgraphs, where one studies which patterns do not appear in a given graph. Among the most basic examples are cluster graphs which forbid paths of length three and which are simply disjoint unions of cliques. Generalizing cluster graphs, cographs and chordal graphs are the first non-trivial classes to study. A graph is a \emph{cograph} if and only if it does not contain an induced path of length four, and there is extensive literature on graph decompositions starting precisely with cographs~\cite{graph_classes}. 

\emph{Chordal graphs} are exactly those graphs which do not contain any induced cycle of length at least four. Chordal graphs are of importance for an abundance of reasons in diverse areas of graph theory, as reflected by the fact that they are known under numerous names, including triangulated graphs~\cite{berge1966some,golumbic}, perfect elimination graphs~\cite{golumbic_goss} or decomposable graphs~\cite{cowell_dawid_lauritzen_spiegelhalter}. The structure of chordal graphs allows us to efficiently solve various algorithmic problems that are in general difficult, such as graph coloring or finding the largest clique which can be solved in polynomial time for chordal graphs~\cite{gavril, parra_scheffler}.

In the case of power graphs, forbidden induced subgraphs were considered  in~\cite{manna2021forbidden}, where the authors classify the finite groups whose power graph is a split graph (for power graphs, this is characterized by forbidding the induced $4$-cycle and its complement) or a threshold graph (forbidding the $4$-cycle and its complement as well as the $4$-vertex path). In~\cite{cameron2021finite}, the same authors started to characterize the finite groups whose power graph is a cograph. The major contribution of~\cite{cameron2021finite} is a complete classification of simple groups whose power graph is a cograph, relative to very natural number theoretic oracles. Recently we completed this work by giving a classification of arbitrary non-solvable groups whose power graph is a cograph~\cite{bk}, relative to the same number theoretic oracles.

In the present work we continue with the investigation of cycles and chordality in power graphs. 
Originally, the question of which groups admit chordal power graphs was posed in~\cite{cameron2021graphs}. In~\cite{manna2021forbidden} the authors give a classification of nilpotent groups with chordal power graph. There is also an earlier study on groups whose power graph does not contain induced cycles of length four~\cite{doostabadi}.
Hereafter we say that a group is \textit{power-chordal} if its power graph is a chordal graph.

For our first main result, we carry out a case-by-case analysis in the class of finite simple groups. We obtain a complete classification of power-chordal finite simple groups, relative to typical number theoretic oracles.
\begin{thm}\label{thm:main}
	If $G$ is a finite simple group, then the power graph of $G$ is chordal if and only if
	$G$ is isomorphic to one of the following:
	\begin{enumerate}
		\item a cyclic group of prime order,
		\item an alternating group $\alt(5)$, $\alt(6)$ or $\alt(7)$, or the group $\psl_3(4)$,
		\item $\psl_2(q)$ with $q\geq 5$ a prime power, where each of $(q-1)/\gcd(q-1,2)$ and $(q+1)/\gcd(q+1,2)$ have at most two prime divisors and at most one prime divisor with multiplicity greater than $1$,
		\item $\sz(q)$ with $q=2^{2n+1}$ for $n\geq 1$, where each of $q-1$, $q-2^{n+1}+1$, and $q+2^{n+1}+1$ have at most two prime divisors and at most one prime divisor with multiplicity greater than $1$.
	\end{enumerate}
\end{thm}

The general approach we take is similar to~\cite{cameron2021finite} in the sense that chordality in power graphs induces similar, albeit weaker, number theoretic restrictions compared with those used in~\cite{cameron2021finite} to classify simple groups whose power graphs are cographs. To deal with the remaining cases, where the number theoretic restrictions are inconclusive, we devise several sufficient conditions for chordality and non-chordality of power graphs.
 
Beyond simple groups, we observe connections between long cycles and long paths in power graphs of arbitrary groups. Not only are chordal power graphs the next natural class to consider but we also transition from forbidding small, explicitly given subgraphs to infinite families of forbidden subgraphs. Indeed, while threshold graphs, split graphs or cographs can be defined by forbidding at most three fixed graphs, no finite set of forbidden induced subgraphs is restrictive enough to define chordal graphs. In the context of power graphs, the situation of infinitely many forbidden subgraphs has not been investigated in the literature before.
This challenges us to develop generic methods, dealing with types of graphs rather than a finite set of explicitly given structures. 

Maybe even more importantly, chordality in graphs encodes global information. In the case of cographs it was sufficient to consider local restrictions, by which we mean restrictions that can be expressed in terms of subgraphs of bounded size, and which translate to group theoretic properties of centralizers of group elements.
In contrast, here we are confronted with the problem that a cycle locally resembles a path and thus we might expect that similar methods would not be sufficient in the case of chordal graphs. 

Somewhat surprisingly, we do show that there is a global bound on the length of the longest induced path in a non-chordal power graph. In turn, chordal power graphs can actually be characterized by excluding finitely many graphs.
\begin{thm}\label{thm:bound_max_len}
	Let $G$ be a finite power-chordal group. If $\wp$ is an induced path on $L$ vertices in $\pow(G)$, then $L\leq 19$. In particular, if $G$ is any finite group, then $\pow(G)$ is chordal if and only if it neither contains induced cycles of length $\leq 20$ nor contains an induced path of length $20$.
\end{thm}
The result is obtained by analyzing the action of $\aut(G)$ on the power graph: in the absence of long cycles, paths of maximal length restrict possible automorphisms more severely with increasing length.

Our final contribution is a characterization of direct factors of power-chordal groups. Spelling out the details turns out to be somewhat involved, so we only give a compact reformulation here. For the full picture, see Theorem~\ref{lem:direct_prod} in Section~\ref{sec_direct_prod}.
\begin{thm}[Theorem~\ref{lem:direct_prod} rephrased]
	If $G$ is power-chordal, then $G$ is directly indecomposable or $G$ can be explicitly characterized in terms of basic group theoretic properties.
\end{thm}
As a consequence we obtain the following lemma, which provides a basis for advancing the classification of power-chordal groups beyond the simple case.
\begin{lem}
	If $G$ is a finite power-chordal group, then $\soc(G)$ is either isomorphic to $C_q^m$ for a prime $q$, or a non-abelian finite simple group $T$ from Theorem~\ref{thm:main}, or a direct product of one of these groups with $C_p$ for a prime $p\neq q$. If $\soc(G)\cong T\times C_p$, then $p$ divides $|T|$ or $T$ is additionally a simple $\mathrm{EPPO}$-group.
\end{lem}

Let us point out that $\mathrm{EPPO}$-groups have been classified by Brandl~\cite[Theorem 2]{bannuscher_tiedt}.

\section{Preliminaries}\label{sec:prelims}
\subsection{Notation}
In numeric expressions, we usually denote the \emph{greatest common divisor} of two integers $a$ and $b$ by $(a,b)$. 

The \emph{power graph} of a group $G$ is the simple graph $\pow(G):=(G,E)$ where $\{x,y\}\in E$ if and only if $x^m=y$ or $y^m=x$ holds for
some $m\in\mathbb{N}$. The \emph{directed power graph} D$\pow(G)$ is the same as the undirected graph, but edges are replaced by arcs pointing from elements to their powers (i.e., D$\pow(G)$ may contain bidirectional edges). In contrast, the \emph{commuting graph} of $G$, denoted by $\mathrm{Com}(G)$, has an edge joining $x$ and $y$ if and only if $xy=yx$ holds in $G$.

Given any graph $\Gamma$ and a subset $M$ of its vertices, the \emph{subgraph induced on $M$} in $\Gamma$ is denoted by $\Gamma[M]$. The \emph{length} of a path (resp. cycle) is the number of its vertices. For example, we use the terms \emph{$4$-vertex path}, \emph{$4$-path} and \emph{path of length $4$} interchangeably.

The \emph{alternating group} and the \emph{symmetric group} of degree $n$ are denoted $\alt(n)$ and $\sym(n)$, respectively.

A group $G$ is called an \emph{$\mathrm{EPPO}$-group}, if all elements of $G$ have prime-power order.

\subsection{Catalan's conjecture}
In dealing with number theoretic problems we repeatedly encounter Catalan's conjecture on two consecutive perfect powers, which was proposed in $1844$ by Catalan, and proved in $2002$ by Mih\u{a}ilescu. 

\begin{thm}[Mih\u{a}ilescu~\cite{mihailescu}]
The only solution to the equation $x^a-y^b=1$ in the natural numbers for $a,b > 1$ and $x,y >0$ is $x=3$, $a=2$, $y=2$ and $b=3$.
\end{thm}

\subsection{Finite simple groups}\label{sec:prelims_simple_grps}
A significant portion of the present paper is dedicated towards a case-by-case analysis of families of finite simple groups and their power graphs based on the classification of finite simple groups (CFSG). A general reference including detailed descriptions of simple groups, their subgroups, and related constructions can be found in~\cite{wilson_2009}. Together with the~\textsc{ATLAS}~of Finite Groups~\cite{atlas} this will be our standard reference for properties of specific groups. We recall important aspects from the theory of finite simple groups that are necessary to lay out our high-level approach.

\begin{thm}[CFSG]\label{cfsg}
	Each finite simple group belongs in one of the following families:
	\begin{enumerate}
		\item cyclic groups of prime order.
		\item alternating groups $\alt(n)$ with $n\geq 5$.
		\item classical groups of Lie type. These are the simple groups of type $A_n(q)=\psl_{n+1}(q)$,
		\\$C_n(q)=\mathrm{PSp}_{2n}(q)$,~$\negphantom{x}\phantom{x}^2A_2(q)=\mathrm{PSU}_{n+1}(q)$,~$B_n(q)=\mathrm{P\Omega}_{2n+1}(q)$,~$D_n(q)=\mathrm{P\Omega}^+_{2n}(q)$,
		\\and~$\negphantom{x}\phantom{x}^2D_n(q)=\mathrm{P\Omega}^-_{2n}(q)$.
		\item exceptional groups of Lie type and the Tits group.
		The exceptional groups of Lie type are~$\sz(q)=\negphantom{x}\phantom{x}^2B_2(q)$ with $q\geq 8$ an odd power of $2$,
		$\mathrm{Ree}(q)=\negphantom{x}\phantom{x}^2G_2(q)$ with $q\geq 9$ an odd power of $3$, $G_2(q)$, $\negphantom{x}\phantom{x}^3D_4(q)$, $F_4(q)$, $\negphantom{x}\phantom{x}^2F_4(q)$, $E_6(q)$, $\negphantom{x}\phantom{x}^2E_6(q)$, $E_7(q)$, and $E_8(q)$.
The Tits group is $\negphantom{x}\phantom{x}^2F_4(2)'$.
		\item 26 sporadic groups. These are the Mathieu groups $\mathrm{M}_{11}$, $\mathrm{M}_{12}$, $\mathrm{M}_{22}$, $\mathrm{M}_{23}$, $\mathrm{M}_{24}$, the Janko groups $\mathrm{J}_1$, $\mathrm{J}_2$, $\mathrm{J}_3$, the Conway groups $\mathrm{Co}_1$, $\mathrm{Co}_2$, $\mathrm{Co}_3$, the Fischer groups $\mathrm{Fi}_{22}$, $\mathrm{Fi}_{23}$, $\mathrm{Fi}_{24}'$, the Higman-Sims groups $\mathrm{HS}$, the McLaughlin group $\mathrm{McL}$, the Held group $\mathrm{He}$, the Rudvalis group $\mathrm{Ru}$, the Suzuki group $\mathrm{Suz}$, the O'Nan group $\mathrm{O'N}$, the Harada-Norton group $\mathrm{HN}$, the Lyons group $\mathrm{Ly}$, the Thompson group $\mathrm{Th}$, the Baby Monster group $\mathrm{B}$ and the Fischer-Griess Monster group $\mathrm{M}$.
	\end{enumerate}
\end{thm}
If $G$ is a finite simple group of Lie type, then $G=L_n(q)$ or $G=\negphantom{x}\phantom{x}^dL_n(q)$ with $d\in\{2,3\}$ and $L\in\{A,B,C,D,E,F,G\}$. A group is said to be a \emph{version of $L_n(q)$} (respectively~$\negphantom{x}\phantom{x}^dL_n(q)$) if it is a central quotient of the universal group of type $L_n(q)$ (respectively~$\negphantom{x}\phantom{x}^dL_n(q)$). For example, the universal group of type $A_n(q)$ is $\spl_{n+1}(q)$ and the simple version of $A_n(q)$ is $\psl_{n+1}(q)$ (see~\cite{gorenstein1994classification} for a detailed discussion of versions of groups of Lie type). We call versions of $L_n(q)$ \emph{untwisted} and versions of $\negphantom{x}\phantom{x}^dL_n(q)$ with $d\in\{2,3\}$ \emph{twisted}.

For each of the (families of) finite simple groups much is known about their structure, usually including a classification of (maximal) subgroups and detailed descriptions of the centralizers they contain. We take this information mostly from~\cite{gorenstein1994classification}, ~\cite{wilson_2009}, and~\cite{atlas}.

Inspired by~\cite{manna2021forbidden}, we treat ``large" simple groups of Lie type with a uniform argument by identifying copies of central quotients of $\spl_3(q)$ inside of other simple groups. 

We employ a well-known construction of subgroups of simple groups via their associated Dynkin diagrams, where subgroups can be generically associated with sub-Dynkin diagrams. 
\begin{lem}\label{spl3_levi_factor}
	Let $G=L_n(q)$ be a finite simple group of Lie type and assume that $n$ is at least $3$. 
	Furthermore assume that $G$ is either untwisted or one of $\negphantom{x}\phantom{x}^2E_6(q)$ or $\negphantom{x}\phantom{x}^2D_n(q)=\mathrm{P\Omega}^-_{2n}(q)$ (in the latter case assume $n>3$).
	 Then $G$ contains a subgroup that is isomorphic to $\mathrm{SL}_3(q)$ divided by a group of scalar matrices. Moreover, the groups~$E_6(q)$,~$\negphantom{x}\phantom{x}^2E_6(q)$, $E_7(q)$, and~$E_8(q)$ contain a simple subgroup of type $F_4(q)$.
\end{lem} 
\begin{proof}
	First assume that $G$ is untwisted. Since $n\geq 3$ holds, the Dynkin diagram associated to $G$ contains a sub-diagram of type
	$A_2$ (i.e, a single edge). By~\cite[Proposition 2.6.2]{gorenstein1994classification}, we have that $G=L_n(q)$ then contains a subgroup
	that is a version of $A_2(q)$ (the discussion before~\cite[Proposition 2.6.2]{gorenstein1994classification} shows that we have $n=d_1=1$ in Part b) of the Proposition in the untwisted case). Each version of $A_2(q)$ is a quotient of $SL_3(q)$ by a central subgroup (see~\cite[Theorem 2.2.6 b)]{gorenstein1994classification}), so the first claim holds for untwisted groups. If $G$ is $\negphantom{x}\phantom{x}^2E_6(q)$ or $\negphantom{x}\phantom{x}^2D_n(q)$ with $n>3$, we consider the Dynkin diagram of the corresponding untwisted version:
	
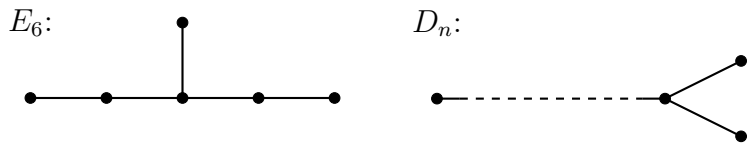
\begin{figure}[h]
		\begin{center}
			\begin{tikzpicture}
				\node[draw=none] at (0,1) {$E_6$:};
				\node[draw=none] at (0,-0.45) {};
				\draw[fill=black] (0,0) circle (2pt);
				\draw[fill=black] (1,0) circle (2pt);
				\draw[fill=black] (2,0) circle (2pt);
				\draw[fill=black] (3,0) circle (2pt);
				\draw[fill=black] (4,0) circle (2pt);
				\draw[fill=black] (2,1) circle (2pt);
				\draw[thick] (0,0) to (1,0);
				\draw[thick] (1,0) to (2,0);
				\draw[thick] (2,0) to (3,0);
				\draw[thick] (3,0) to (4,0);
				\draw[thick] (2,0) to (2,1);
			\end{tikzpicture}	\hspace{15pt}	
			\begin{tikzpicture}
				\node[draw=none] at (0,1) {$D_n$:};
				\draw[fill=black] (0,0) circle (2pt);
				\draw[fill=black] (3,0) circle (2pt);
				\draw[fill=black] (4,-0.5) circle (2pt);
				\draw[fill=black] (4,0.5) circle (2pt);
				\draw[thick] (0,0) to (0.3,0);
				\draw[thick,dashed] (0.3,0) -- (2.7,0);
				\draw[thick] (2.7,0) to (3,0);
				\draw[thick] (3,0) to (4,0.5);
				\draw[thick] (3,0) to (4,-0.5);
			\end{tikzpicture}
		\end{center}
	\end{figure}
	
	In any case, the Dynkin diagram admits an automorphism of order $2$, say $\sigma$, that induces an automorphism of $H\in\{E_6(q^2),
	D_n(q^2)\}$, which we again call $\sigma$. Let $f$ denote the automorphism of $H$ induced by the Frobenius automorphism of $\mathbb{F}_{q^2}$ over $\mathbb{F}_q$ and write $\alpha:=\sigma\circ f$. Then, by definition, $G$ is equal to the set of fixed points $\mathrm{Fix}_\alpha(H)$ of $\alpha$ in $H$. Considering the Dynkin diagram corresponding to $H$, we note that we can always choose a sub-diagram of type $A_2$ that is point-wise fixed by $\sigma$. By definition the corresponding version of $A_2(q^2)$ in $H$, say $S\leq H$, is point-wise fixed by $\sigma$ as well.
	In turn, $G=\mathrm{Fix}_\alpha(H)$ contains a subgroup $S\cap \mathrm{Fix}_\alpha(H)=\mathrm{Fix}_\alpha(S)=\mathrm{Fix}_f(S)$ that
	is a version of $A_2(q)$, since, by choice of $S$ and $f$, we have that $f$ acts on $S$ as the corresponding Frobenius automorphism (see~\cite[Theorem 2.1.2, Proposition 2.1.10]{gorenstein1994classification}). That finishes the proof of the first claim.

	Towards the second claim, consider the group $F_4(q)$. By~\cite{gorenstein1994classification}, $F_4(q)$ can be defined as the set of fixed points of the diagram automorphism in $E_6(q)$. Thus, by
	the construction of $\negphantom{x}\phantom{x}^2E_6(q)$ discussed above, we see a copy of $F_4(q)$ inside of $\negphantom{x}\phantom{x}^2E_6(q)$, given by the fixed points of the Frobenius automorphism acting on $F_4(q^2)\leq E_6(q^2)$. Both $E_7(q)$ and $E_8(q)$ contain a version of $E_6(q)$, which can again be seen in terms of Dynkin sub-diagrams, and thus also a version of $F_4(q)$. The version of $E_6(q)$ has a center of order $1$ or $3$, so the corresponding version of $F_4(q)$ is simple, since for $q>2$, the simple group $F_4(q)$ has trivial Schur-multiplier while $F_4(2)$ has Schur-multiplier $C_2$ (see~\cite[Table 6.1.2]{gorenstein1994classification}). 
\end{proof}

\section{Induced cycles in power graphs}
We begin by recalling that induced paths, and thus also cycles, in the power graph have an alternating nature. This observation is well-known and can be found in~\cite{cameron2021finite}, for example. To reveal the alternating structure we have to consider the directed power graph.
\begin{lem}\label{lem:paths_alternating}
	Let $\wp$ be an induced path of length at least $3$ in $\pow(G)$, then the directions of consecutive arcs in $\mathrm{D\!\pow(G)}[\wp]$ alternate.
\end{lem}
\begin{proof}
	If $x,y,z$ are consecutive vertices in $\wp$ such that $x$ points towards $y$ and $y$ points towards $z$ in the
	directed power graph, then, by definition, there would also be an arc pointing from $x$ towards $z$, contradicting
	that $\wp$ is an induced path.
\end{proof}

We formalize this in terms of the following definition.

\begin{dfn}
	Let $G$ be a finite group and let $\wp$ be an induced path in $\pow(G)$. Then each vertex contained in $\wp$
	has only in-edges or only out-edges in the directed graph induced by $\wp$ in $\mathrm{D\!\pow(G)}$. Accordingly we say vertices are
	\emph{in-vertices} or \emph{out-vertices} of $\wp$ and we write $\wp^-$ (respectively $\wp^+$) to denote the set of 
	in-vertices (respectively out-vertices) of $\wp$. The notion naturally applies for induced cycles as well.
\end{dfn} 

The notion of in-vertices allows us to iteratively transfer cycles into reduced forms.
\begin{dfn}\label{def:power-reduced}
We say a path (resp. cycle) in $\pow(G)$ is \emph{power-reduced} if all in-vertices have prime order.
\end{dfn}

\begin{lem}\label{lem:primes_alternate}
	If $\mathfrak{C}$ is a power-reduced cycle in $\pow(G)$, then consecutive in-vertices have distinct prime orders.
\end{lem}
\begin{proof}
	If $x$ and $y$ are consecutive in-vertices in $\mathfrak{C}$, then they are joined via a common out-vertex $z$. By definition, $x,y\in \langle z\rangle$ holds and $x$ and $y$ have prime orders. Since $x$ and $y$ are not joined via an edge, they must have distinct orders.
\end{proof}
\begin{lem}\label{lem:assume_4_cycle_pow_red}
	If $\mathfrak{C}$ is an induced $4$-vertex cycle in $\pow(G)$, then there is also a power-reduced $4$-vertex cycle
	induced in $\pow(G)$.
\end{lem}
\begin{proof}
	The two in-vertices of $\mathfrak{C}$, say $x$ and $y$, are not joined in $\pow(G)$, despite lying in a common cyclic overgroup. Thus, their orders must be incomparable, i.e., there are distinct primes $p$ and $q$ such that $p$ divides $|x|$ more often than $|y|$ and $q$ divides $|y|$ more often than $|x|$. We can then replace $x$ by a power of $x$ of order $p$ and replace $y$ by a power $y$ of order $q$ and still obtain an induced $4$-cycle.	
\end{proof}

\begin{lem}\label{lem:assume_cycle_pow_red}
	Let $\mathfrak{C}$ be an induced cycle with at least $4$ vertices in $\pow(G)$. Then $\pow(G)$ contains a power-reduced induced cycle with at least $4$ vertices.
\end{lem}
\begin{proof}
	There is nothing to show for $4$-cycles due to Lemma~\ref{lem:assume_4_cycle_pow_red}. Assume that $\mathfrak{C}$ has more than $4$ vertices. Since
	in- and out-vertices alternate in $\mathfrak{C}$, there must be at least $6$ vertices and $3$ out-vertices.	
	
	Let $g\in\mathfrak{C}^-$ and replace it by $g'=g^{\frac{|g|}{p}}$ of prime order $p$ to obtain another induced subgraph $\mathfrak{C}'$.
	Then, by definition, edges $\{x,g\}$ incident with $g$ in $\mathfrak{C}$ induce edges $\{x,g'\}$ in $\mathfrak{C}'$. If
	$\mathfrak{C}'$ is not a cycle, then there is some vertex $x$ of $\mathfrak{C}\setminus\{g\}$ that is not joined with $g$ in $\pow(G)$
	but joined with $g'$. But then, since $\mathfrak{C}'$ has at least $6$ vertices, the only way to not have an induced cycle on more than $3$ vertices in $\mathfrak{C}'$ (in which case we repeat the argument with a smaller cycle) is that
	$g'$ is joined with every $x\in\mathfrak{C}\setminus\{g\}$. Since we assume $\mathfrak{C}'$ to contain at least $3$ out-vertices and
	consecutive in-vertices cannot both have their order be a power of $p$ (otherwise they would be joined by definition, since they lie in a common cyclic overgroup), there must be two out-vertices, say $h_1$ and $h_2$, that are joined with $g'$ and such that their common in-vertex, say $x$, has an order distinct from $p$. Let $x'$ be a power of $x$ of order $q$ for a prime distinct from $p$.
	Then $\pow(G)[\{g',h_1,x',h_2\}]$ forms an induced $4$-cycle with the required property.
\end{proof}

We point out that the reduction may in general produce shorter cycles, with $4$-cycles forming a lower bound.
There is also a correspondence between power-reduced paths and cycles in $\pow(G)$ and certain paths and cycles in the commuting graph $\mathrm{Com}(G)$.
\begin{lem}\label{lem:connection_pow_com}
	Let $M\subseteq G$ be a set of group elements. If $\pow(G)[M]$ is a power-reduced path (resp. cycle), then
	$\mathrm{Com}(G)[M^-]$ is a path (resp. cycle) in the commuting graph, in which consecutive vertices have distinct prime orders.
	On the other hand, such a path (resp. cycle) gives rise to a power-reduced path (resp. cycle) in $\pow(G)$ by
	joining consecutive vertices, say $x$ and $y$, via a common out-vertex $xy$. 
\end{lem}

\subsection{The four vertex cycle}
Given the results of the previous section, it makes sense to briefly take a closer look at cycles with exactly four vertices. The following lemma is a reformulation of~\cite[Lemma 3.1]{doostabadi}, adapted to our specific use case.

\begin{lem}\label{lem:char_prop_C4}	
	Assume $G$ does not contain elements whose order is divisible by $pqr$ or $p^2q^2$ for distinct primes
	$p$, $q$ and $r$.
	Then the power graph of $G$ has an induced subgraph forming a $4$-cycle if and only if
	there exist two elements $g$ and $h$ in $G$ with the following properties:
	\begin{enumerate}
		\item we have $|g|=|h|=q^m$ for a prime $q$ and $m>1$,
		\item $g^q=h^q$, but $\langle g\rangle\neq\langle h\rangle$
		\item there is a prime $p\neq q$ that divides $|C_G(\langle g,h\rangle)|$.
	\end{enumerate}		
\end{lem}

A natural example of the situation described in the previous lemma would be the group $Q_8\times C_3$, where a pair of generators for $Q_8$ can be taken to be $g$ and $h$ above. Of course there is, in general, no need for the group generated by an induced $4$-vertex cycle to resemble this structure. However, the following restriction always applies to $4$-cycles.
\begin{cor}\label{cor:4_cycles_high_order}
	Let $\mathfrak{C}$ be an induced $4$-cycle in the power graph of a finite group $G$. Then the order of each out-vertex is divisible by $p^2q$ for distinct primes $p$ and $q$.
\end{cor}

In the case of $4$-vertex cycles we also note the following structural implication.
\begin{cor}\label{cor:simple-chordal_max_subgr}
	If $\pow(G)$ contains an induced $4$-vertex cycle, say $\mathfrak{C}=(g_1,g_2,g_3,g_4)$, then $U:=\langle g_1,\dots,g_4\rangle$ admits a proper normal cyclic subgroup. In particular, if $G$ is simple then there is a maximal subgroup of $G$ that contains $\mathfrak{C}$.
\end{cor}

Lastly we note that there is a complete list of nilpotent groups with induced $4$-vertex cycles in their power graph given in~\cite{doostabadi}.

\subsection{Larger cycles \& chordality}\label{sec_direct_prod}
\begin{dfn}
	A graph is \emph{chordal} if it does not admit any induced cycles on at least four vertices. 
	We say that a group $G$ is \emph{power-chordal} if its power graph is chordal.
\end{dfn}	
A result of Dilworth~\cite{dilworth_1950} shows that power graphs are perfect so in particular power graphs do not admit cycles of odd length. 
We collect general restrictions on element orders in power-chordal groups. They were also observed in~\cite{manna2021forbidden}.
\begin{lem}
	If $G$ contains elements of order $p^2q^2$ or $pqr$ for pairwise distinct primes $p,q,r$, then $G$ is not power-chordal.
	If $G$ is an $\mathrm{EPPO}$-group, i.e., all elements of $G$ have prime power order, then $G$ is power-chordal.
\end{lem}

For nilpotent groups there is a complete description of power-chordal groups available.
\begin{thm}[{\cite[Theorem 17, Theorem 19]{manna2021forbidden}}]\label{thm:pocho_nilpot}
	If $G$ is nilpotent, then the power graph of $G$ is chordal if and only if $G$ is a $p$-group or $G=C_{q^m}\times P$, where $P$ is a group of exponent $p$ and $p$ and $q$ are distinct primes.
\end{thm}

Beyond the nilpotent case not much is currently known. One problem in deciding if a power graph is chordal or not lies in the non-locality of the property: it is not clear how to decide if a given group element $g\in G$ lies in a, potentially long, cycle by only looking at a small portion of group elements. In the following we provide sufficient conditions for chordality and non-chordality of power graphs based on properties of centralizers of (pairs of) elements in $G$. They provide effective tools for our purpose of classifying power-chordal groups later on.

\begin{lem}\label{lem:sufficient_condition_chordal}
	If $G$ contains elements of order $p^2q^2$ for distinct primes $p,q$, then $\pow(G)$ is non-chordal. Otherwise, 
	If for each $x$ in $G$ of prime order $p$ the centralizer $C_G(x)$ is either a $p$-group or of the form $C_{q^n} \rtimes P$ where $q$ is a prime distinct from $p$ and $P$ is a $p$-group, then $\pow(G)$ is chordal. 
\end{lem}

\begin{proof}
	Assume that $\pow(G)$ is not chordal, so there exists an induced $t$-cycle $\mathfrak{C}=(g_1,\dots,g_t,g_{t+1}=g_1)$ in $\pow(G)$ with $t>3$. By assumption we also have that no element of $G$ has an order that is divisible
	by $p^2q^2$ for distinct primes $p$ and $q$. Moreover, if $G$ contains an element whose order is divisible by three distinct primes, then the centralizer of this element is neither a $p$-group or of the form $C_{q^n} \rtimes P$, so assume otherwise.
	
	By Lemma~\ref{lem:assume_cycle_pow_red}, we may assume without loss of generality that $\mathfrak{C}$ is power-reduced.
	Let $g_i$ be an in-vertex of $\mathfrak{C}$, so $g_i\in\langle g_{i\pm 1}\rangle$ and denote the (prime) order of $g_i$ by $p$. Then $g_{i\pm 1}$ is not a $p$-element. Let $h_{i\pm 1}\in\langle g_{i\pm 1}\rangle$ be a $p'$-element. Then $C_G(g_i)$ contains both $h_{i-1}$ and $h_{i+1}$, showing the claim if $h_{i-1}$ and $h_{i+1}$ are not connected via an edge in $\pow(G)$ and not equal. On the other hand, if $h_{i-1}$ and $h_{i+1}$ are connected via an edge in $\pow(G)$ or equal, then the set $\{h_{i+1},g_{i+1},g_{i-1},g_i\}$ induces a cycle of length $4$ in $\pow(G)$.
	In this case the claim follows from Lemma~\ref{lem:char_prop_C4}.
\end{proof}

To give a first example, we apply the previous lemma to generalized dihedral groups.

\begin{lem}
	A generalized dihedral group $A \rtimes C_2$, where $A$ is abelian and $C_2$ acts on $A$ by inversion, is power-chordal if and only if $A$ is power-chordal. 
\end{lem}
\begin{proof} Let $G = A \rtimes C_2$.
	If $A$ is an abelian power-chordal group, then it is isomorphic to $C_{q^n} \times C_p^m$ where $n,m \in \mathbb{N}_0$ and $p$ and $q$ are distinct primes by Theorem~\ref{thm:pocho_nilpot}. 
	If $x\in A$ then $C_G(x)=C_A(x)=A$ and $A$ satisfies the centralizer condition in  Lemma~\ref{lem:sufficient_condition_chordal}. The other elements in $G$ have order $2$ and generate their own centralizers as a subgroup of $G$, so they also satisfy the centralizer condition.
\end{proof}

Moreover, we want to note that generalized quaternion groups fail to fulfill the third condition of Lemma~\ref{lem:char_prop_C4} and are indeed power-chordal whenever the maximal normal cyclic subgroup is.
\begin{lem}
	The generalized quaternion group of order $4n$ is given by
	\[
		Q_{4n}=\langle x,y\mid x^{2n}=y^4=1, x^n=y^2, x^y=x^{-1}\rangle.
	\]It is power-chordal if and only if $\langle x\rangle$ is, i.e., if and only if we have
	$2n=2^ap^b$ for some odd prime $p$ and $a,b\in\mathbb{N}_0$ with $a\leq 1$ or $b\leq 1$.
\end{lem}
\begin{proof}
	Note that, for power-chordality of $Q_{4n}$, it is necessary that $|Q_{4n}|$ has at most two distinct prime divisors, and so $|x|$ has at most two prime divisors. If $|x|=2^a$, then $Q_{4n}$ is a $2$-group and power-chordal by Theorem~\ref{thm:pocho_nilpot}. Otherwise, $p^b$ divides $|x|$ for some odd prime $p$. By definition, $y$ normalizes $\langle x \rangle$. All elements of odd order in $Q_{4n}$ are contained in $\langle x \rangle$ and have centralizers equal to $\langle x \rangle$. The other elements have even order and act on $\langle x \rangle$ by inversion, so their centralizers are $2$-groups. Thus, assuming that either $a$ or $b$ is at most $1$, Lemma~\ref{lem:sufficient_condition_chordal} implies chordality of $\pow(Q_{4n})$.
\end{proof}

Next, we give a precise description of direct factors of power-chordal groups in terms of elementary group theoretic notions. In all cases, possible direct factors of chordal groups are very restricted but depending on their exact structure, we are left to deal with a number of situations. We point out that the results match those found in~\cite[Theorem 3.6-3.8]{doostabadi}, whenever the theorems apply (the authors give one-sided restrictions in the case of forbidden $4$-cycles assuming that the group has non-trivial center). 
\begin{thm}\label{lem:direct_prod}
	Let $G=H\times K$ be a non-trivial direct product with $H$ and $K$ power-chordal. Then, up to interchanging roles of $H$ and $K$, $G$ is power-chordal if and only if $G$ is in one of the following cases.
	\begin{enumerate}
		\item $H$ has exponent $p$ for a prime $p$, but $H$ is not cyclic. Then $G$ is power-chordal if and only if 
		$K=\langle w\rangle\rtimes P$ with an element $w$ of order $q^n$ for a prime $q\neq p$ and $n\in\mathbb{N}$, such that
		that $P$ is a $p$-group and $C_P(w^{q^{n-1}})$ is trivial or of exponent $p$. 		
		
		\item $H$ is cyclic of order $p^nq^m$ with distinct primes $p$ and $q$ and $n\in\{0,1\}$.
		Then $G$ is power-chordal if and only if one of the following holds:
		\begin{enumerate}
			\item $n=1$, $m>1$, and $K$ has exponent $p$,
			\item $n=m=1$, $K$ is an $\mathrm{EPPO}$-group and all its elements have order $p$ or $q$,
			\item $n=0$, $m>1$, and $K$ is an $\mathrm{EPPO}$-group with exponent $q^nb$ where $n\in\mathbb{N}_0$ and $b$ is square-free,
			\item $n=0$, $m=1$, and the following conditions are fulfilled: elements of $K$ have prime power order or order $qp^m$ with $m\in\mathbb{N}$.
		Cyclic $p$-subgroups $Z_1$ and $Z_2$ of $K$ only intersect in $\{1\}$, $Z_1$ or $Z_2$ and
		each $q$-element in $K$ centralizes at most one cyclic $p$-subgroup of a given order.	
		\end{enumerate}	
		\item $H=N\rtimes U$ with $N\cong C_{p^m}$ and $U\cong C_{q^n}$, such that
		$C_U(\soc(N))=\{1\}$ holds, with distinct primes $p$ and $q$, and $m,n\geq 1$.
		Then $G$ is power-chordal if and only if one of the following holds:
		\begin{enumerate}
			\item $K$ has exponent $q$,
			\item $m=1$ and $K$ is a cyclic $q$-group,
			\item $K=N'\rtimes U'$ with $N'\cong C_{r^d}$ and $U'\cong C_{q^f}$, for $r$ a prime distinct from $p$ and $q$, $d\geq 1$, and $C_{U'}(\soc(N'))=\{1\}$ holds. If $m>1$ then $d=1$ and $f\in\{0,1\}$ and if $n>1$ then $d=1$.
			
			\item $K=N'\rtimes U'$ with $N'\cong C_{p^dq^e}$ and $U'\cong C_{q^f}$, for $e\in\{0,1\}$ and $d\geq 1$, and $C_{U'}(\soc(N'))=\{1\}$ holds. If $e=1$ then  $m=n=d=1$. If $e=0$ then $m>1$ implies $f\leq 1$ and $n>1$ implies $d=1$.
		\end{enumerate}
		\item $G$ has prime power order.
	\end{enumerate}
\end{thm}
\begin{proof}
	\begin{enumerate}
		\item By assumption there exist elements $h$ and $h'$ in $H$ that generate distinct cyclic subgroups. Let $k\in K$ be of order
		coprime to $p$ (otherwise, $p$-groups are handled in Case 4). If $k$ does not generate a unique cyclic subgroup of order $|k|$, then there is another element  of order coprime to $p$, say $k'$, which is not joined
		with $k$ in $\pow(K)$. Then there is an induced $8$-cycle with in-vertices $(h,k,h',k')$ (the out-vertices can be taken to be the products of consecutive in-vertices). Thus, there can be at most one cyclic subgroup of $K$ of a fixed order coprime to $p$ and so
		the elements of $K$ of orders coprime to $p$ generate a normal cyclic subgroup $Z$. However, if $G$ is chordal then $G$ does not contain
		elements whose order has three distinct prime divisors, so $Z$ is either trivial or a $q$-group for some prime $q\neq p$. By definition, $K/Z$ is a $p$-group and furthermore $C_{K/Z}(\soc(Z))$ has exponent $p$ or otherwise $G$ contains subgroups isomorphic to $C_p\times C_{p^2}\times C_q$ and is then not power-chordal by Theorem~\ref{thm:pocho_nilpot}. Thus, the given restrictions are necessary. To see that they are sufficient, assume that $\mathfrak{C}$ is a power-reduced induced cycle in $G$ despite $G$ fulfilling the assumptions. It is sufficient to consider power-reduced cycles by Lemma~\ref{lem:assume_cycle_pow_red}.
		If $\mathfrak{C}$ has length greater than $4$, then by Lemma~\ref{lem:primes_alternate} there must be distinct in-vertices of order $q$, which is clearly not possible. Thus, $\mathfrak{C}$ has length $4$. By Lemma~\ref{lem:char_prop_C4} there exist two elements $g$ and $h$ in $G$ of order $r^m$ for a prime $r\in \{p,q\}$ and $m>1$ such that $g^r=h^r$ but $\langle g \rangle \neq \langle h \rangle$. By assumption, here $r$ can only be equal to $p$, but then the out-vertices of $\mathfrak{C}$ must have orders divisible by $p^2q$ and such element orders do not exist in $G$.	
		 
		 \item \begin{enumerate}
		 	\item For chordality of $G$, the restriction $\exp(K)=p$ is necessary and sufficient by Theorem~\ref{thm:pocho_nilpot}. 
		\item Again, the given restrictions are necessary by Theorem~\ref{thm:pocho_nilpot}.
		For the other direction, assume the restrictions hold. Then there is a unique cyclic subgroup of order $p$ (resp. $q$) that centralizes
		$q$-elements in $K$ (resp. $p$-elements in $K$). Thus, there cannot be a power-reduced cycle of length greater than $4$ in $\pow(G)$, since its in-vertices would need to centralize two distinct cyclic subgroups of coprime order, so its in-vertices cannot have a non-trivial component in $K$.
		Induced $4$-cycles are excluded by the non-existence of elements of order greater than $pq$ (cf. Lemma~\ref{lem:char_prop_C4}). The claim follows by Lemma~\ref{lem:assume_cycle_pow_red}.
		\item The given restrictions are necessary by Theorem~\ref{thm:pocho_nilpot}.
		For the other direction, if $y$ is an in-vertex of an induced power-reduced cycle in $\pow(G)$ and $|y|\neq q$,
		then the cycle cannot have more than $4$ vertices. This follows from the fact that, since $K$ is $\mathrm{EPPO}$, $y$
		centralizes a unique $C_q$ in $G$. By Lemma~\ref{lem:char_prop_C4} we can exclude power-reduced $4$-cycles. If $g$ and $h$ in $G$ are of order $q^m$ and commute with a $p$-element for a prime $p$ distinct from $q$, then $g$ and $h$ live in $H$ and generate the same cyclic subgroup.
		
		\item If $K$ contains elements of different orders than the ones stated, then $G$ contains subgroups of the form
		$C_{pr}\times C_{p^2}$ or $C_{pqr}$ with pairwise distinct primes $p,q,r$ and is thus not power-chordal.
		If $K$ contains cyclic $p$-subgroups $Z_1$ and $Z_2$ that intersect in a proper subgroup of both $Z_1$ and $Z_2$,
		then there is an induced $4$-cycle in $\pow(G)$ on the set $\{x,xz_1,xz_2,z\}$, where $z_i$ generates $Z_i$, $z$ generates $Z_1\cap Z_2$ and $x$ generates $H$. If $K$ contains a $q$-element $y$ that centralizes two distinct cyclic $p$-subgroups $\langle
		z_1\rangle$ and $\langle z_2\rangle$ with $|z_1|=|z_2|$,
		then there is an induced $8$-cycle with in-vertices $\{x,z_1,y,z_2\}$. Thus, all conditions listed above are necessary. 

Assume now that they are all fulfilled but $G$ is not power-chordal. Then there is an induced cycle $\mathfrak{C}$ in $\pow(G)$ and we may assume
		that it is power-reduced by Lemma~\ref{lem:assume_cycle_pow_red}. Then, without loss of generality, we may further assume that each out-vertex in $\mathfrak{C}$ is the product of the two in-vertices it joins. Since $K$ is assumed to be power-chordal, $\mathfrak{C}$ cannot be fully contained in $K$, thus there is some in-vertex $y$ of order $q$ that is not contained in $K$. Since by Lemma~\ref{lem:primes_alternate} 
		orders of consecutive in-vertices in $\mathfrak{C}$ are distinct, $y$ centralizes two distinct cyclic $p$-subgroups generated by
		the in-vertices that are at distance $2$ from $y$. By assumption, $y$ must then be contained in $H$, otherwise its centralizer would be too small. But the same argument works for all in-vertices of order $q$ in $\mathfrak{C}$ and since $H\cong C_q$ holds,
		there is no other in-vertex of order $q$ so $\mathfrak{C}$ is a $4$-cycle. To form a $4$-cycle, the out-vertices of $\mathfrak{C}$ must generate distinct cyclic subgroups of $G$ that intersect in an element of order $p$ (cf. Lemma~\ref{lem:char_prop_C4}), a contradiction to the assumptions.
		 \end{enumerate}
		
		\item Write $H=N\rtimes U$ with $N\cong C_{p^m}$.\begin{enumerate} 
			\item If $K$ has exponent $q$, then  all elements of order $p^iq$ in $G$ lie in $N\times K$. Moreover $N\times K$ is power-chordal by Theorem~\ref{thm:pocho_nilpot}. If $G$ is not power-chordal, then there is a power-reduced cycle in $\pow(G)$ by Lemma~\ref{lem:assume_cycle_pow_red}. We note that all of the out-vertices have order $p^iq$ for an appropriate $i$ and thus live in $N\times K$, a contradiction. 
			\item If $K$ is a cyclic $q$-group, then all elements of order $pq^i$ in $G$ lie in $N\times K$. Since $m=1$, the latter is power-chordal as in Part 3a) and so is $G$.
			\item
			Assume that $G$ fulfills the assumptions and contains a power-reduced cycle $\mathfrak{C}$. If the length of $\mathfrak{C}$ is greater than $4$, then, since $G$ contains unique maximal cyclic $p$-subgroups and $r$-subgroups, respectively, $\mathfrak{C}^-$ must contain an element $v$ of order $q$ and the in-vertices of $\mathfrak{C}$ that are
			at distance $2$ from $v$ must be of order $p$ and $r$. But no element in $G$ commutes with
			a $p$-element and an $r$-element simultaneously. If the length of $\mathfrak{C}$ is $4$, then by Lemma~\ref{lem:char_prop_C4}, one of $H$ or $K$ must contain distinct Sylow $q$-subgroups that
			intersect non-trivially. Since we assume $C_U(\soc(N))=C_{U'}(\soc(N'))=\{1\}$, Sylow $q$-subgroups of $H$ (or $K$, respectively) only intersect trivially, a contradiction.
			\item 
			The $q$-elements of $H$ centralize at most one unique $C_p$ in $G$, so they cannot be in-vertices of power-reduced cycles of length greater than $4$ (a similar argument was used in Part 2c). 
	
			If $e=0$, then the same argument applies with roles of $K$ and $H$ interchanged, so no cycles of length greater than $4$ exist. We can exclude $4$-cycles using the exact same argument that was used in Part 3c).

			If $e=1$, then by the other assumptions $G$ only contains elements of order $p$, $q$ or $pq$, so by Lemma~\ref{lem:char_prop_C4} there are no $4$-cycles in $\pow(G)$. Write $K=N'\rtimes U'$ where $N'\cong C_{pq}$. Then $q$-elements of $K$ that are not contained in $N'$ also centralize a unique $C_p$ in $G$ (namely $N$) and as before, they cannot be in-vertices of power-reduced cycles of length greater than $4$. In conclusion, the in-vertices of order $q$ in a power reduced cycle in $\pow(G)$ would all be contained in $N'$, but $N'$ has a unique maximal cyclic $q$-subgroup and then there can be at most one such in-vertex, contradicting the fact that there are no induced $4$-cycles.
		\end{enumerate}
		It remains to prove that one of these cases occurs. Since $H$ has cyclic subgroups of orders $p^m$ and $q^n$, the restrictions of Part 2) apply to $K$ with respect to $p$ and $q$, depending on the values of $m$ and $n$. Moreover, by definition $H$ contains distinct cyclic subgroups of order $q$,
		so by the same argument as in the proof of Part 1), we have that $K$ is either a $q$-group, or $K$ contains a normal cyclic $r$-subgroup for a prime $r$ distinct from $q$, whose quotient is a $q$-group or trivial.
		
		In the former case, $K$ is a $q$-group. If $\exp(K)\neq q$ then, by Theorem~\ref{thm:pocho_nilpot}, $m=1$ and $K$ is cyclic.
		In the latter case, write $K=\langle w\rangle\rtimes Q$ with a $q$-group $Q$ and $|w|=r^d$ with $d\geq 1$. If $Q$ is trivial, there is nothing to show, so assume otherwise.
		If $r\neq p$, then $G$ contains a subgroup of the form 
		$C_{p^m}\times C_r\times C_Q(w^{r^{d-1}})$, so $C_Q(w^{r^{d-1}})$ is trivial by Theorem~\ref{thm:pocho_nilpot} and if $m>1$ or $n>1$ holds, then $d=1$ must hold. Moreover, $Q$ is isomorphic to a subgroup of $\aut(C_r)$ and thus cyclic, say $Q\cong C_{q^f}$ and if $m>1$ then $f$ is at most $1$.
		
		If $r=p$, then $G$ contains a subgroup of the form 
		$C_{p^m}\times C_p\times C_Q(w^{p^{d-1}})$, so $C_Q(w^{p^{d-1}})$ is either $C_{q^e}$ for some $e\geq 1$ or trivial. Moreover, $Q/C_Q(w^{p^{d-1}})$ is isomorphic to a subgroup of $\aut(C_p)$ and thus cyclic.
		In any case, $K$ is a subgroup of $(C_{p^d q^e}\rtimes C_{q^f})$ as claimed in Part 2c). To see that $C_Q(w^{p^{d-1}})$ splits
		from (or is equal to) $Q$, assume otherwise. There is some $x\in Q$ that generates $Q/C_Q(w^{p^{d-1}})$, since $\aut(\langle w^{p^{d-1}}\rangle)$ is cyclic of order $p$.
		If $Q$ is not a split extension of $\langle x\rangle$ by $C_Q(w^{p^{d-1}})$, then $x^i\in C_Q(w^{p^{d-1}})$ for some $i$ such that
		$x^i\neq 1$ holds. But $x$ does not centralize $\langle w^{p^{d-1}}\rangle$, so there are conjugates of $x$ that generate a cyclic subgroup distinct from $\langle x\rangle$ in $K$. These cyclic subgroups intersect non-trivially, so they give rise to a $4$-cycle in
		$K\times C_p$.
		
		We deduce restrictions on $m,n,d,e$, and $f$. There are subgroups of the form $C_{q^n}\times C_{p^d}\times C_{q^e}$ and
		$C_{p^m}\times C_{p^d}\times C_{q^e}$ in $G$ (recall that we assume $K$ is not a $q$-group, so $d\geq 1$). By Theorem~\ref{thm:pocho_nilpot}
		we have $e\in\{0,1\}$ and if $e=1$ holds, then $m$, $n$ and $d$ must be $1$. Furthermore, if $e=1$ then $f\in \{0,1\}$ or otherwise
		$K$ contains distinct cyclic subgroups of order $q^2$ that intersect non-trivially, giving rise to an induced $4$-cycle in $\pow(G)$.
		If $e=0$, then we still have subgroups of $G$ isomorphic to $C_{p^m}\times C_{q^f}$ and $C_{q^n}\times C_{p^d}$, so
		if $m>1$, then $f\in\{0,1\}$ and if $n>1$, then $d=1$.
		
		\item Recall that groups of prime power order are always power-chordal by Theorem~\ref{thm:pocho_nilpot}. We need to argue here that
		this is the only possibility left.
		To this end, assume that neither $H$ nor $K$ is cyclic or of prime exponent. We show that one of $H$ and $K$ can then be placed in Part 3) of the present Lemma. For appropriate primes $p$ and $q$ (not necessarily distinct) there exist distinct cyclic $p$-subgroups
		$Z_1$ and $Z_2$ of $H$ and distinct cyclic $q$-subgroups $C_1$ and $C_2$ of $K$. If $p\neq q$ holds, then $G$ is not power-chordal and we can find an induced $8$-cycle among all products of elements in these cyclic subgroups. Thus, if $r$ is a prime distinct from $p$, then
		both $H$ and $K$ have unique maximal cyclic subgroups whose orders are $1$ or divisible by $r$. We may assume that $G$ is not of prime power order, so without loss of generality, $H$ contains elements of order coprime to $p$. The uniqueness of maximal cyclic $r$-subgroups for primes $r\neq p$ implies that all elements of order coprime to $p$ generate a normal cyclic subgroup $Z$ of $H$.
		By Theorem~\ref{thm:pocho_nilpot}, $Z$ has order $r^m$ for some prime $r$ and $|H|$ is only divisible by $p$ and $r$.
		So $H=\langle w\rangle\rtimes P$ with $|w|=r^m$ and a $p$-group $P$. 
		
		Now $K$ is not a $p$-group, otherwise $K$ would have exponent $p$ by Theorem~\ref{thm:pocho_nilpot} (we assume that $K$ is not cyclic here). But if $K$ is not a $p$-group, the same structural implications we deduced for $H$ also apply to $K$, say $K=\langle w_1\rangle\rtimes P_1$. Then $G$ contains a subgroup of the from $\langle w\rangle\times C_P(w)\times\langle w_1\rangle\times C_{P_1}(w_1)$ which, by Theorem~\ref{thm:pocho_nilpot}, is only power-chordal
		if one of the centralizers is trivial. But then we are in Part 3).
	\end{enumerate}
\end{proof}

As an application of Theorem~\ref{lem:direct_prod} we explore the possible socles of power-chordal groups and thereby restrict the structure of arbitrary power-chordal groups.

\begin{lem}\label{lem:socles}
	If $G$ is power-chordal, then $\soc(G)$ is isomorphic to $C_p^m\times C_q$, $C_p^m$, $T$, or $C_p\times T$ with a non-abelian simple group $T$, $m\in\mathbb{N}$, and distinct primes $p$ and $q$. In the last case, if $p$ does not divide $|T|$, then $T$ is a simple $\mathrm{EPPO}$-group. 
\end{lem}
\begin{proof}
	Write $\soc(G)=A\times T_1\times \dots\times T_m$ with an abelian group $A$ and non-abelian simple groups $T_i$.
	If $G$ is power-chordal, then $m\leq 1$ holds. Otherwise there would be elements $x_1,x_2\in T_1$ of order $p$ and $y_1,y_2\in T_2$ of order $q$ with distinct primes $p$ and $q$, such that $\langle x_1\rangle\neq\langle x_2\rangle$ and $\langle y_1\rangle\neq\langle y_2\rangle$ hold. Since $x_i$ commutes with $y_j$ for all combinations of $i$ and $j$, the four elements constructed above give rise to an $8$-cycle in $\pow(G)$.
	Since a non-abelian simple group $T$ does not admit normal cyclic subgroups, for each prime $p$ dividing $|T|$ we have that
	$T$ contains distinct cyclic subgroups of order $p$. Thus, $C_p^2\times T$ is never power-chordal.
	The rest follows from Theorem~\ref{lem:direct_prod}.
\end{proof}

In Section~\ref{sec:prelims_simple_grps} we give a classification of simple power-chordal groups. Thus, if $p$ does divide $|T|$, then the pairs $(T,p)$ are classified and one could in principle work out the combinations where $C_p\times T$ is
power-chordal. We give an example by discussing the alternating groups. All elements of $\mathrm{Alt}(5)$ have order $2$, $3$, or $5$, so
$C_p\times \mathrm{Alt}(5)$ is power-chordal for all choices of $p$ (even for $p>5$). However, $\mathrm{Alt}(6)$ contains distinct cyclic subgroups of order $4$ that intersect in a $C_2$, so $C_p\times \mathrm{Alt}(6)$ can only be power-chordal for $p=2$, in which case it is indeed power-chordal as we checked using \textsc{GAP}~\cite{GAP4}.
This implies that $\mathrm{Alt}(7)\times C_p$ can only be power-chordal if $p=2$, which it again is. Later we show that $\mathrm{Alt}(n)$ with $n\geq 8$ is not power-chordal (see Lemma~\ref{lem:chordality_an_sn} for concrete arguments).

\subsection{Building paths in the power graph}
Before we turn to simple groups, we give two sufficient conditions for non-chordality of power graphs. In both cases, the general idea is that finiteness of $G$ makes it necessary that paths in $\pow(G)$ are not extendable beyond a certain point or eventually can be extended to a cycle.
This in turn induces restrictions on the structure of $G$. The first restriction is concerned with centralizers of elements in $G$ and complements Lemma~\ref{lem:sufficient_condition_chordal}. 

\begin{lem}
	If $G$ is power-chordal, then each non-singleton connected component of the Gruenberg-Kegel graph of $G$ contains a prime $p$ such that
	the following holds: $G$ contains an element $x$ with $|x|=p$ and $C_G(x)=Z\rtimes P$, where $Z$ is a cyclic $q$-group for a prime $q\neq p$
	and $P$ is a $p$-group. Moreover $C_{P}(Z)$ is either cyclic or of exponent $p$.
\end{lem}
\begin{proof}
	If $G$ is not an $\mathrm{EPPO}$-group, i.e., if the Gruenberg-Kegel graph has a non-trivial component, then by definition there exist paths of length $3$ in $G$.
	 
	We say that a group element $x\in G$ is an \emph{end point} of an induced path $\wp$ in $\pow(G)$, if
	$x$ is an in-vertex of $\wp$, such that $x$ has degree $1$ in $\wp$ and such that there is no way to extend $\wp$ beyond $x$ to another
	in-vertex. If end points do not exist at all, then each path can be extended indefinitely and ultimately has to be contained in an induced cycle since $|G|$ is finite. 

On the other hand, if $x$ is an end point of $\wp$ and the latter has length at least $3$, then
	there is another in-vertex $y$ in $\wp$ that has distance $2$ from $x$. We may assume that $x$ has prime order, say $p$, or otherwise
	we can replace $x$ by some appropriate power of $x$ (if this does not define another path, $G$ cannot be power-chordal).
	If $C_G(x)$ contains distinct cyclic $q$-groups of the same order, generated by $y_1$ and $y_2$ say, then
	only one of these can contain or be contained in $\langle y\rangle$, say $y_1$, and then $\wp$ can be extended beyond $x$ via $xy_2$ and $y_2$, contradicting the choice of $x$. 

A similar argument rules out more prime divisors in $C_G(x)$ apart from $p$ and $q$ (here, $q$ must appear since $y$ cannot be a $p$-element if $\wp$ forms a path).
	Thus we have $C_G(x)=Z\rtimes P$ where $Z$ is a cyclic $q$-group and $P$ is a $p$-group. The restrictions on $C_{P}(Z)$ follow immediately from Theorem~\ref{lem:direct_prod}. 

Finally, if $\{p',q'\}$ is any edge of the Gruenberg-Kegel graph, then by definition there exist
	commuting elements $x'$ and $y'$ of order $p'$ and $q'$, respectively. There is an induced $3$-vertex path on $\{x',x'y',y'\}$
	in $\pow(G)$, which can be extended until some end point is reached. But each prime divisor of any element order appearing in such an extended path still belongs to the component of $p'$ and $q'$.
\end{proof}

The second sufficient condition is purely defined in terms of pairs of group elements and their centralizers. As such it is easily verifiable,
while still being able to produce cycles of arbitrary length, which makes it more likely to be effective. For example, many cycles in simple groups emerge from the following lemma.

\begin{lem}\label{lem:conjug_cyclic_property_sufficient}
	If $G$ contains elements $x$ and $y$ of coprime orders such that
	$xy=yx$ holds but neither $\langle x\rangle$ is normal in $C_G(y)$ nor is $\langle y\rangle$ normal in $C_G(x)$, then $\pow(G)$ is non-chordal.
\end{lem}

\begin{proof}
	By assumption, there is a $C_G(x)$-conjugate $y'$ of $y$ and a $C_G(y)$-conjugate $x'$ of $x$ such that
	$\langle y\rangle\neq\langle y'\rangle$ and $\langle x\rangle\neq\langle x'\rangle$ hold. Since $x$ and $y$ have coprime orders, $\pow(G)$ contains an induced path of the form $(x,xy,y,yx',x')$. Since $x'$ is conjugate to $x$, we can now inductively continue this path. More precisely, $x'$ must admit another conjugate of $y$ in its own centralizer $C_G(x')$, say $y''$, such that $y''$ and $y$ generate distinct cyclic subgroups of $G$. Then $y''$ is conjugate to $y$ and thus the path keeps growing by repeating the argument for $y''$ (if $y''\in\langle y'\rangle$ we are done immediately). But $G$ is finite and so eventually this process must produce a cycle in $\pow(G)$.
\end{proof}

We first apply the previous lemma to show that special linear groups over finite fields are in most cases not power-chordal. This result is also used later to rule out infinitely many simple groups while enumerating those with chordal power graph (cf. also~\cite[Theorem 5.5]{cameron2021finite}, where a similar argument was used in the classification of simple power-cograph groups).

\begin{lem}\label{lem:sl3q}
	If $q \not \in \{2,4\}$ then $\spl_3(q)$ contains a cycle of length greater than $3$. Furthermore, such a cycle can be chosen not to contain scalar matrices.  	
\end{lem}
\begin{proof}
	Let $p$ be prime such that $q=p^m$ for some $m\in\mathbb{N}$.
	We find elements $X$ and $Y$ in $\spl_3(q)$ of orders $p$ and $q-1$ which satisfy Lemma~\ref{lem:conjug_cyclic_property_sufficient}.  
	For $n\in \mathbb{N}$ define
	\begin{eqnarray*} 
		X_n := \begin{bmatrix}
		1 & n & n \\
		0 & 1 & 0 \\
		0 & 0 & 1 
		\end{bmatrix}\text{ and }
		Y_n := \begin{bmatrix}
		y^n & 0 & 0 \\
		0 & y^n & y^n-y^{-2n} \\
		0 & 0 & y^{-2n} 
		\end{bmatrix},
	\end{eqnarray*}
	where $y$ is of order $q-1$ in $\mathbb{F}_q^*$. Let $X:=X_1$ and $Y:=Y_1$ and note that $X^n=X_n$ and $Y^n=Y_n$ .
	Indeed, the order of $X$ is $p$ and the order of $Y$ is $q-1$.
	Straightforward calculations show that the centralizer of $Y$ in $\spl_3(q)$ is 
	$$\left \{\begin{bmatrix}
		a & b & b \\
		c & d & d-e \\
		0 & 0 & e 
		\end{bmatrix} \ \middle | \ e\in \mathbb{F}_q^*, \det \begin{bmatrix}
		a & b  \\
		c & d 
		\end{bmatrix} = e^{-1}\right \}$$
	and the centralizer of $X$ in $\spl_3(q)$ is 
	$$\left \{\begin{bmatrix}
		a & b & c \\
		0 & d & d-a^{-2} \\
		0 & a-d & a+a^{-2}-d 
		\end{bmatrix} \ \middle | \ a\in \mathbb{F}_q^*, b,c,d \in \mathbb{F}_q\right \}.$$
	Let $$M_Y:=\begin{bmatrix}
		1 & 0 & 0 \\
		1 & 1 & 0 \\
		0 & 0 & 1 
	\end{bmatrix}\in C_{\spl_3(q)}(Y), 
		M_X:=\begin{bmatrix}
		1 & 1 & -1 \\
		0 & 0 & -1 \\
		0 & 1 & 2 
	\end{bmatrix}\in C_{\spl_3(q)}(X).$$
	We have $$M_Y X M_Y^{-1}:=\begin{bmatrix}
		0 & 1 & 1 \\
		-1 & 2 & 1 \\
		0 & 0 & 1 
	\end{bmatrix},
	M_X Y M_X^{-1}:=\begin{bmatrix}
		y & -2(y-y^{-2}) & 0 \\
		0 & y^{-2} & 0 \\
		0 & y-y^{-2} & y 
	\end{bmatrix}.$$

	If $q\neq 2,4$ then $M_Y$ does not normalize $\langle X \rangle$ and $M_X$ does not normalize $\langle Y \rangle$. This still holds modulo any subgroup $S$ of scalar matrices. In particular, Lemma~\ref{lem:conjug_cyclic_property_sufficient} applies to $XS$ and $YS$ in $\spl_3(q)/S$.
\end{proof}

Since the lemma does not deal with the cases $q=2$ and $q=4$, we are left to deal with $\spl_3(2)$ and $\spl_3(4)$ separately.
\begin{lem}
	$\spl_3(2)$ is power-chordal. $\spl_3(4)$ is not power-chordal.
\end{lem}
\begin{proof}
	Recall that $\spl_3(2)\cong \psl_3(2)\cong \psl_2(7)$, which is an $\mathrm{EPPO}$-group and hence power-chordal. 
	We have $Z(\spl_3(4))\cong C_3$, so $\spl_3(4)$ is not power-chordal if we can find two distinct copies of $C_4$ that intersect in a $C_2$ (see Lemma~\ref{lem:char_prop_C4}). Let us write $\mathbb{F}_4=\mathbb{F}_2[\omega]$, then we can take
	$\begin{bmatrix}
		1 & w & 0 \\
		0 & 1 & 1 \\
		0 & 0 & 1 
	\end{bmatrix}\text{ and }
	\begin{bmatrix}
		1 & w & 1 \\
		0 & 1 & 1 \\
		0 & 0 & 1 
	\end{bmatrix}$ as suitable generators.
\end{proof}

\section{Simple groups with chordal power graph}\label{sec:simples}
We successively consider the non-abelian families from the classification of finite simple groups as given in Section~\ref{sec:prelims_simple_grps}.

\subsection{Symmetric \& alternating groups}
We find long induced cycles in symmetric and alternating groups for larger degrees.
\begin{lem}\label{lem:chordality_an_sn}
	The power graph of $\sym(n)$ is chordal if and only if $n\leq 5$. The power graph of $\alt(n)$ is chordal if and only if $n\leq 7$.
\end{lem}
\begin{proof}
	Let $a=(1,2)$, $b=(3,4,5)$, $c=(1,6)$, $d=(2,3,4)$, $e=(1,5)$, and $f=(3,4,6)$. Then $(a,ab,b,bc,c,cd,d,de,e,ef,f,fa,a)$ is an induced $12$-cycle in $\sym(6)$. If we multiply each transposition by $(7,8)$, we obtain that $\alt(8)$ is not power-chordal. Since power-chordal groups are closed under taking subgroups, we find that $\sym(n)$ is not power-chordal for $n> 5$ and $\alt(n)$ is not power-chordal for $n>7$. On the other hand, we checked that $\sym(5)$ and $\alt(7)$ are power-chordal via computations in GAP~\cite{GAP4} and SageMath~\cite{sagemath}.
In case of $\sym(5)$ we could explicitly check for chordality of the power graph. For $\alt(7)$, recall that centralizers
of elements of order $4$, $5$, and $7$ are no bigger than the cyclic subgroups generated by those elements. This implies
that $\pow(\alt(7))$ does not contain induced $4$-cycles (see Lemma~\ref{lem:char_prop_C4}) and if the power graph does contain
a long cycle, we might assume that it is power-reduced by Lemma~\ref{lem:assume_cycle_pow_red}. By exclusion of other element orders, the orders of in-vertices of such a cycle would have to alternate between $2$ and $3$. But each element of order $2$ centralizes a unique element of order $3$ in $\alt(7)$, so this is impossible. Finally, symmetric and alternating groups of smaller degree are power-chordal since the property is subgroup-closed.
\end{proof}

\subsection{Classical groups of Lie type}
The classical groups of Lie type are typically split into finer families of groups consisting of special linear, orthogonal, symplectic, or unitary groups. We follow the ordering given in Theorem~\ref{cfsg}.

The first groups we consider are projective special linear groups. We split the analysis into two parts: projective special linear groups $\psl_2(q)$ and $\psl_n(q)$ for $n\geq 3$. Depending on the parity of $q$, we give number theoretic conditions characterizing when the former groups have chordal power graphs.

\begin{lem}\label{lem:psl_2_necessary} Assume that $\pow(\psl_2(q))$ is chordal.
	\begin{enumerate}
		\item If $q$ is even, then $q-1$ and $q+1$ each have at most two prime divisors and at most one prime divisor with multiplicity greater than $1$. 
		\item If $q$ is odd, then $(q-1)/2$ and $(q+1)/2$ each have at most two prime divisors and at most one prime divisor with multiplicity greater than $1$.  
	\end{enumerate}
\end{lem}

\begin{proof}
If $q$ is even, then $\psl_2(q)$ contains cyclic subgroups of orders $q-1$ and $q+1$. If $q$ is odd, then $\psl_2(q)$ contains cyclic subgroups of orders $\frac{q-1}{2}$ and $\frac{q+1}{2}$ (cf.~\cite[Section 3.10]{wilson_2009}).
\end{proof}

In~\cite{cameron2021finite}, similar number theoretic conditions are derived for those groups isomorphic to $\psl_2(q)$ whose power graph is a cograph. By comparing the conditions, we note that the class of power-chordal groups isomorphic to $\psl_2(q)$ includes those with cograph power graphs. Moreover, giving precise solutions to the number theoretic restrictions in~\cite{cameron2021finite} is expected to be hard, so the same holds for Lemma~\ref{lem:psl_2_necessary}.

Yet, relative to the number theoretic conditions we obtain a complete characterization in the following lemma.
\begin{lem}
	Let $G:=\psl_2(q)$ and assume that $q$ fulfills the necessary conditions from the conclusion of Lemma~\ref{lem:psl_2_necessary}. Then $\pow(G)$ is chordal.
\end{lem}	
\begin{proof}
	The maximal cyclic subgroups of $\psl_2(q)$ have orders $q$, $q^-:=(q-1)/(2,q-1)$ and $q^+:=(q+1)/(2,q+1)$~\cite[Section 3.10]{wilson_2009}. In particular, under the necessary conditions, $\psl_2(q)$ does not contain elements of order $p^2p'^2$ or $pp'p''$ for distinct primes $p$, $p'$ and $p''$. Our aim is to use Lemma~\ref{lem:sufficient_condition_chordal} to show chordality of $\pow(\psl_2(q))$, assuming that the cyclic subgroups of order $q^-$ and $q^+$ are each power-chordal. 

To this end, consider an element $x\in G$ of prime order $p$. If $p$ divides $q$, then the centralizer of $x$ is a $q$-group by the list of maximal cyclic subgroups that $G$ contains. Assume that $p$ does not divide $q$. Going through the list of all subgroups of $\psl_2(q)$ (see~\cite{huppert}), 
subgroups of $G$ that are not of prime power order are either cyclic groups, or dihedral groups of order
$2m$, where $m$ divides $q^-$ or $q^+$, or they are isomorphic to $A_4$, $S_4$, $A_5$, $C_{q_0^d}\rtimes C_t$ where $q_0$ is the unique prime dividing $q$, or $\psl_2(q')$, where $q'$ divides $q$. Since $C_G(x)$ has $x$ in its center, $C_G(x)$ must either be a $p$-group, or cyclic (with at most two distinct prime divisors by our assumptions on $q$),
or dihedral of order $2m$. In the latter case, again using the fact that $Z(C_G(x))$ is non-trivial, we note that $p=2$ holds and $m$ must be even. Then $C_G(x)$ is either a $2$-group or a semi-direct product $C_{p'^d}\rtimes P$ with a $2$-group $P$ and a prime $p'\neq 2$. In any case, $C_G(x)$ fulfills the condition from Lemma~\ref{lem:sufficient_condition_chordal}, and since $x$ was chosen arbitrarily, $G$ is power-chordal.
\end{proof}

It remains to consider the case of dimension at least $3$, and here we show that only two small examples of special linear groups are power-chordal, namely $\psl_3(2)\cong\psl_2(7)$ and the $\mathrm{EPPO}$-group $\psl_3(4)$.

\begin{lem}\label{lem:psl_larger_3}
	Let $G:=\psl_n(q)$ with $n\geq 3$ and an arbitrary prime power $q$. Then $\pow(G)$ is chordal if and only if
	$G$ is isomorphic to $\psl_3(2)\cong\psl_2(7)$ or $\psl_3(4)$.
\end{lem}
\begin{proof}
	We first treat the case $n=3$ and then we treat the case $n>3$ in one go.
	\paragraph{Case $n=3$:} By~\cite[Section 3.10]{wilson_2009}, $G$ contains subgroups isomorphic to $C_{q+1}$ and $C_{q-1}\times C_{(q-1)/(3,q-1)}$. So if $q$ is odd, then either $q-1$ equals $6$ (hence $q=7$), or a $2$-power, or $G$ is not power-chordal according to Theorem~\ref{thm:pocho_nilpot}. In the $2$-power case, $(q+1)=2^a+2=2(2^{a-1}+1)$ for some exponent $a$ and $q=2^a+1$. Then the pair $(q-1,q)$ is a solution to Catalan's conjecture~\cite{mihailescu}, so either $q=9$ or $q$ is prime.
	 If $q$ is prime, then $a$ must be a power of $2$ as well, so $a-1$ is odd and then $q+1$ is divisible by $6$. If $C_{q+1}$ is power-chordal, then $q+1$ is two times a $3$-power because if $q-1=2^a$ and $q+1\geq 6$ then $2$ divides $q+1$ exactly once. So $(q-1)=2(3^b-1)$ for some exponent $b$ and $3^b-1$ must be a $2$-power. By Catalan's conjecture~\cite{mihailescu} $q=17$ or $q=5$. In these cases and furthermore for $q=7$ and $q=9$, we  check with GAP~\cite{GAP4} that none of these groups are power-chordal by using Lemma~\ref{lem:conjug_cyclic_property_sufficient}.

	  If $q$ is even, then $C_{q-1}\times C_{(q-1)/(3,q-1)}$ is chordal only if $q-1$ is a prime power or three times a prime (Theorem~\ref{thm:pocho_nilpot}). Furthermore, $G$ also contains subgroups of the form $C_{q+1}\times C_{(q-1)/(3,q-1)}$ (there is a cyclic group of order $q^2-1$ in $\mathrm{GL}_2(q)$ and for even $q$ we have $\gcd(q-1,q+1)=1$). Assume for now that $q\neq 4$. Thus, $q+1$ has to be a prime power (we always have $\gcd(q-1,q+1)=1$ for even $q$). Again due to Catalan's conjecture~\cite{mihailescu}, if $q$ is not $2$ or $8$, then $q+1$ is even forced to be prime. But if $q+1=2^a+1$ is prime, then $a$ is a power of $2$ and $q-1=2^a-1$ is divisible by the three distinct primes dividing $2^8-1$ whenever $a\geq 8$. So for $a\geq 8$, we obtain a contradiction since $q-1$ can have at most two distinct prime divisors. The remaining cases are $q\in\{2,4,8,16\}$. We checked them for power-chordality via GAP~\cite{GAP4} by using Lemma~\ref{lem:conjug_cyclic_property_sufficient}. The power-chordal exceptions are $\psl_3(2)\cong\psl_2(7)$, and $\psl_3(4)$, which is an $\mathrm{EPPO}$-group.
	 
	  \paragraph{Case $n>3$:} If $n>3$ there is a subgroup of $\psl_n(q)$ isomorphic to $\spl_3(q)$, which is not power-chordal if $q\neq 2,4$ (Lemma~\ref{lem:sl3q}). But $\psl_4(2)\cong A_8$ which is not power-chordal by Lemma~\ref{lem:chordality_an_sn}, and $\psl_4(4)$ contains $A_5\times A_5$ which is not power-chordal: we can easily build cycles in $\pow(A_5\times A_5)$ by considering distinct cyclic subgroups of order $3$ in one copy of $A_5$ and distinct subgroups of order $5$ in the other copy, for example.
\end{proof}

Next, we examine projective symplectic groups. These groups exist in even dimension at least $2$, but we need only consider dimension at least $4$ since in dimension $2$, the groups coincide with projective special linear groups (cf.~\cite{wilson_2009}). We prove that none of these groups have chordal power graphs.

\begin{lem}\label{lem:psp}
	Let $G:=\mathrm{PSp}_{2n}(q)$, where $n\geq 2$. Then $G$ is not power-chordal.
\end{lem}
\begin{proof}
	We treat the two cases $n=2$ and $n>2$ separately.
	\paragraph{Case $n=2$:}
	Let $n=2$. If $q$ is even and not equal to $2$, then there exist subgroups of the form $C_{q\pm 1}\times C_{q\pm 1}$ (cf.~\cite[Section 3.10]{wilson_2009}). Then both $q-1$ and $q+1$ must be prime powers, otherwise there are nilpotent non-chordal subgroups (Theorem~\ref{thm:pocho_nilpot}). By Catalan's conjecture~\cite{mihailescu}, we have $q\in\{2,4,8\}$ and we check these cases using GAP~\cite{GAP4}.
	If $q=8$, then $G$ contains a subgroup isomorphic to $\psl_2(8)\times D_{18}$, which is not power-chordal: $D_{18}$ contains distinct subgroups of order $2$ and $\psl_2(8)$ contains distinct subgroups of order $7$, giving rise to an $8$-cycle in the power graph.	 
	 If $q=4$ then $G$ contains a subgroup isomorphic to $S_6$, which is not power-chordal. If $q=2$ then $G$ itself is isomorphic to $S_6$.
	  
	  If $q$ is odd, then there exist subgroups of the form $C_{q\pm 1}\times C_{(q\pm 1)/2}$ (cf.~\cite[Section 3.10]{wilson_2009}). In this case, $(q-1)/2$ and $(q+1)/2$ now have to be prime powers, so one of them is a power of $2$ since one of them is divisible by $2$. Then one of $q-1$ and $q+1$ must be equal to $4$ or otherwise there is a subgroup of the form $C_4\times C_2\times C_p$ with an odd prime $p$. So the remaining cases are $q=3$ and $q=5$.
	  If $q=5$ then $G$ contains a maximal subgroup isomorphic to $S_3\times S_5$, which is not power-chordal.
	  If $q=3$ then $G$ contains a subgroup isomorphic to $S_6$, which is not power-chordal.	
	  
	  \paragraph{Case $n>2$:}
	  If $n>2$, then $\mathrm{PSp}_{2n}(q)$ contains subgroups isomorphic to $\spl_3(q)$ divided by a group of scalar matrices (cf. Lemma~\ref{spl3_levi_factor}), which are not power-chordal by Lemma~\ref{lem:sl3q} if $q\neq 2,4$. Moreover, $\mathrm{PSp}_6(4)$ contains $\mathrm{PSp}_{6}(2)$ and the latter contains $S_6$, which is not power-chordal according to Lemma~\ref{lem:chordality_an_sn}. For $n>3$, the group $\mathrm{PSp}_{2n}(q)=\mathrm{Sp}_{2n}(q)$ with $q\in \{2,4\}$ contains $\mathrm{PSp}_{6}(q)=\mathrm{Sp}_{6}(q)$ (see~\cite[Theorem 3.7]{wilson_2009}).
\end{proof}

Next in line are projective simple unitary groups and we consider them in dimension at least $3$, because like the projective symplectic groups, they coincide with projective special linear groups in smaller dimensions (cf.~\cite{wilson_2009}).

\begin{lem}\label{lem:psu}
	Let $G:=\psu_n(q)$, where $n\geq 3$. Then $G$ is power-chordal if and only if $n=3$ and $q=2$ hold, in which case $G$ is not simple.
\end{lem}
\begin{proof}
	First consider the case $n=3$. If $q$ is odd, then $\psu_3(q)$ contains cyclic subgroups of order
	$(q-1)(q+1)/\gcd(q+1,3)$, see~\cite[Section 3.10]{wilson_2009}. Both $(q-1)$ and $(q+1)/\gcd(q+1,3)$ are even, so for chordality we need
	$(q-1)(q+1)/\gcd(q+1,3)=2^ap^\delta$ to hold for some $a>1$, $\delta\in\{0,1\}$ and an odd prime $p$. If $\delta=0$ then clearly 
	$q\in\{3,5\}$. Assume $\delta=1$. Then one of $(q-1)$ and $(q+1)/\gcd(3,q+1)$ is a $2$-power. If $(q-1)=2^b$ then
	also $q=2^b+1$ must be prime or $9$ (Catalan's conjecture~\cite{mihailescu}). Assume that $q\not\in \{3,9\}$. Now $(q+1)=2(2^{b-1}+1)$, where $2^{b-1}+1$ is either an odd prime 
	or three times an odd prime. In the former case, we have that $q=2^b+1$ and $2^{b-1}+1$ must both be prime, which is impossible. In the latter case, $G$ also contains subgroups isomorphic to $C_{q+1}\times C_{(q+1)/(q+1,3)}$ (see for example~\cite{montanucci_giovanni}), which is not power-chordal if $q+1$ is divisible by $6p$ for some odd prime $p$.
	
	If $q$ is even, then $q-1$ and $q+1$ are coprime. Then $G$ contains subgroups isomorphic to $C_{q\pm 1}\times C_{(q+1)/(3, q+1)}$.
	Suppose $q\neq 2$. If $3$ divides $q+1$ then we find a subgroup isomorphic to $C_{(q+1)/3}\times \psl_2(q)$. This can be seen as follows: consider $ \mathrm{SU}_3(q)$ as a subgroup of $\spl_3(q^2)$ where $X\in \spl_3(q^2)$ is an element of $ \mathrm{SU}_3(q)$ if and only if $X\overline{X}^T=1$, where $X\mapsto \overline{X}$ is induced by the Frobenius automorphism of $\mathbb{F}_{q^2}/\mathbb{F}_{q}$. Then $\mathrm{SU}_3(q)$ contains a copy of $\mathrm{SU}_2(q)$ given by 
$$\begin{bmatrix}
		A & 0\\
		0 & 1 
	\end{bmatrix},$$
with $A\in \mathrm{SU}_2(q)$. Recall that $\mathrm{SU}_2(q)$ is isomorphic to $\spl_2(q)$, see~\cite[Section 3.6]{wilson_2009}, and also to $\psl_2(q)$ since $q$ is even. Furthermore, there is a subgroup in $\mathrm{SU}_3(q)$ consisting of elements 
	$$\begin{bmatrix}
		\alpha & 0 & 0\\
		0 & \alpha & 0 \\
		0 & 0 & \beta  
	\end{bmatrix},$$
with $\alpha \overline{\alpha} = 1$ and $\beta = \alpha^{-2}$, which is isomorphic to $C_{q+1}$ and intersects $Z(\mathrm{SU}_3(q))$ in a $C_3$. Together with the copy of $\psl_2(q)$, it forms a direct product and we factor out the central $C_3$ to get to $\psu_3(q)$.  Let $x \in \psu_3(q)$ be an element of order $(q+1)/3$ such that $x$ centralizes a copy of $\psl_2(q)$, say $H$. By considering the maximal subgroups of $\psu_3(q)$ with $q$ even (for example, see~\cite{hartley}) we conclude that $C_G(x)= \langle x\rangle \times H$. 
Consider $t$ an involution in $H$, and note that $\langle t \rangle$ is not normal in $C_G(x)$. By~\cite[Chapter 6, 5.16]{suzuki_gt2} we have that $C_G(t)=Q\rtimes \langle x \rangle$ holds, where $Q$ is a $2$-group with $C_G(Q)\subseteq Q$. Thus $\langle x \rangle$ is not normal in $C_G(t)$ because otherwise it would centralize $Q$. By Lemma~\ref{lem:conjug_cyclic_property_sufficient}, $G$ is not power-chordal.
	If $3$ does not divide $q+1$, then $q+1$ must be prime due to Catalan's conjecture~\cite{mihailescu}. The same arguments shows that $q-1$ must be prime or equal to $1$, a contradiction unless $q\in\{2,4,8\}$.
	 We checked the remaining cases with GAP~\cite{GAP4}. The only case in which $G$ is power-chordal is $q=2$.

	Now consider $n>3$. Recall that $\psu_4(q)$ contains $\mathrm{PSp}_4(q)$ (cf. \cite[Section 3.10]{wilson_2009}), which is not power-chordal. For $n>4$, the group
	$\psu_n(q)$ contains $\mathrm{SU}_4(q)$ which is not power-chordal for $q=2$ by calculations we performed in GAP~\cite{GAP4} and otherwise because it contains a subgroup of the form $\spl_2(q)\times \spl_2(q)$  (cf. \cite[Section 3.10]{wilson_2009}) which is not power-chordal for $q>2$ by Theorem~\ref{lem:direct_prod}.
\end{proof}

The remaining classical groups of Lie type to consider are the orthogonal groups.

\begin{lem}
	If $G$ is a simple orthogonal group of Lie rank at most $2$, then $G$ is isomorphic to one of $\psl_2(q),\psl_2(q^2),\mathrm{PSp}_4(q), \psl_4(q), \psu_4(q)$ and hence power-chordality has already been discussed in previous lemmas. Furthermore, the non-simple group $G=$P$\Omega^+_4(q)\cong\psl_2(q)\times\psl_2(q)$ is power-chordal if and only if $q=2$.	
	If $G$ is a simple orthogonal group of Lie rank at least $3$, then $G$ is not power-chordal.
\end{lem}
\begin{proof}
	According to~\cite[Section 3.11]{wilson_2009}, there are generic isomorphisms P$\Omega_3(q)\cong\psl_2(q)$, P$\Omega^+_4(q)\cong\psl_2(q)\times\psl_2(q)$, P$\Omega^-_4(q)\cong\psl_2(q^2)$, P$\Omega_5(q)\cong\mathrm{PSp}_4(q)$, P$\Omega^+_6(q)\cong\psl_4(q)$, and
	P$\Omega^-_6(q)\cong\psu_4(q)$. The only group not explicitly treated before is P$\Omega^+_4(q)\cong\psl_2(q)\times\psl_2(q)$, but this is not a simple group.
	If $G$ is a simple orthogonal group of rank at least $3$, Lemma~\ref{spl3_levi_factor} shows that $G$ contains $\spl_3(q)$ modulo scalar matrices, and by Lemma~\ref{lem:sl3q}, it follows that $G$ is not power-chordal for $q\neq 2,4$. Now let $q=2$ or $4$. Recall that in characteristic $2$ we have P$\Omega_{2n+1}(q)\cong \mathrm{PSp}_{2n}(q)$, which is not power-chordal by Lemma~\ref{lem:psp} .

For rank greater than $3$, $\mathrm{P}\Omega^+_{2n}(q)$ contains $\mathrm{P}\Omega^+_{2n}(q_0)$ with $q_0^k=q$ and $k$ a prime~\cite[Section 3.10]{wilson_2009}. It suffices to consider the case $q=2$. Then $\mathrm{P}\Omega^+_{2n}(2)=\Omega^+_{2n}(2)$ contains a maximal subgroup isomorphic to $\mathrm{PSp}_{2n-2}(2)$. 
Similarly, for $q=2,4$, we have that $\mathrm{PSp}_{2n-2}(q)$ is a maximal subgroup of $\mathrm{P}\Omega^-_{2n}(q)\cong \Omega^-_{2n}(q)$ (see~\cite[Section 3.8.2]{wilson_2009}). Note that $\mathrm{PSp}_{n}(q)$ with $n\geq 4$ is not chordal by Lemma~\ref{lem:psp}. 
\end{proof}

\subsection{Exceptional groups of Lie type \& the Tits group}
This section is devoted to studying power-chordality of exceptional groups of Lie type and the Tits groups. We prove that except particular Suzuki groups with number theoretic problems involved, the rest of the exceptional groups of Lie type and the Tits group are not power-chordal. 

\begin{lem}
	The Suzuki groups, i.e., $\sz(q)$ with $q=2^{2n+1}$ and $n\geq 0$, are power-chordal if and only if the cyclic subgroups of orders $q-1$, $q+2^{n+1}+1$, and $q-2^{n+1}+1$ are all power-chordal.
\end{lem}
\begin{proof}
	Write $q:=2^{2n+1}$ and $r:=2^n$. For $n=0$, $\sz(2)$ is an $\mathrm{EPPO}$-group and hence chordal. For $n\geq 1$, the maximal cyclic subgroups of $\sz(q)$ are isomorphic to $C_q$, $C_{q-1}$, $C_{q+2r+1}$ and $C_{q-2r+1}$. 
	In particular, if all orders of such groups adhere to the necessary conditions, then $\sz(q)$ does not contain elements whose order is divisible by $(pp')^2$ or $pp'p''$ for distinct primes $p$, $p'$, and $p''$. Suzuki showed~\cite{suzuki_1960} that centralizers of elements in $\sz(q)$ are nilpotent and that the maximal cyclic subgroups are also maximal nilpotent subgroups. In particular, if $x$ is an element of prime order $p$, then
	$x$ is contained in a maximal cyclic subgroup which is then equal to its centralizer. By Lemma~\ref{lem:sufficient_condition_chordal}, $\sz(q)$ is power-chordal under the given assumptions.
\end{proof}

\begin{lem}
	The simple groups $\mathrm{Ree}(q)=\negphantom{x}\phantom{x}^2G_2(q)$, with $q=3^{2n+1}$  and $n\geq 1$, are not power-chordal. Furthermore, $\mathrm{Ree}(3)$ is not power-chordal.
\end{lem}
\begin{proof}
	Set $G:=\negphantom{x}\phantom{x}^2G_2(q)$. We utilize results from~\cite{ward_1966}. The centralizer of an involution in $G$ is isomorphic to $C_2\times \psl_2(q)$, and all involutions in $G$ are conjugate. Let $P$ be a Sylow $3$-subgroup of $G$. Then $P$ has order $q^3$ and contains a subgroup $P_1$ which is elementary abelian of order $q^2$. The normalizer is $N_G(P)=P\rtimes C_{q-1}$. Let $x$ be the involution of $C_{q-1}$. We have that $C_P(x)=C_{P_1}(x) \cong \langle y \rangle$ for some $y \in P_1$ of order $q$. Then there exists an element $z\in P$ of order $3^a$ for some $a$ such that $z\in P_1$ but $z \not \in C_{P_1}(x)$. Since $P_1 \subseteq C_G(y)$ ($P_1$ is abelian), we have that $\langle x \rangle \not \trianglelefteq C_G(y)$ because $z\in C_G(y)$ and $x$ has order $2$. Recall that $y\in C_G(x) \cong C_2\times \psl_2(q)$ so $y$ is contained in the copy of $\psl_2(q)$. Then $\langle y \rangle \not \trianglelefteq C_G(x)$ because $\psl_2(q)$ is simple as $q \geq 9$. The result follows from Lemma~\ref{lem:conjug_cyclic_property_sufficient}. Furthermore, we checked that $\mathrm{Ree}(3)$ is not power-chordal with GAP~\cite{GAP4}.
\end{proof}

\begin{lem}
	The groups $G_2(q)$, $\negphantom{x}\phantom{x}^3D_4(q)$, $\negphantom{x}\phantom{x}^2F_4(q)$, and the Tits group
	$\negphantom{x}\phantom{x}^2F_4(2)'$ are not power-chordal.
\end{lem}
\begin{proof}
	$G_2(q)$ contains $\spl_3(q)$ which is non-chordal except in the cases $q=2$ or $q=4$ by Lemma~\ref{lem:sl3q}. But $G_2(2)$ is $\mathrm{PSU}_3(3)\rtimes C_2$ and also non-chordal and $G_2(4)$ contains $G_2(2)$. Moreover, $\negphantom{x}\phantom{x}^3D_4(q)$ contains $G_2(q)$ and is hence not power-chordal. Finally, $\negphantom{x}\phantom{x}^2F_4(2)'$ contains $\spl_3(3)$ and 
	$\negphantom{x}\phantom{x}^2F_4(q)$ contains $\negphantom{x}\phantom{x}^2F_4(2)'$~\cite{wilson_2009}, so the claim follows since $\spl_3(3)$ is not power-chordal.
\end{proof}

\begin{lem}\label{lem:f4q}
	The groups $F_4(q)$ are not power-chordal.
\end{lem}
\begin{proof}
	 $F_4(q)$ contains $(\mathrm{SU}_3(q)\times \mathrm{SU}_3(q))\rtimes2$ when $q\equiv 0,1$ mod $3$ and $(\spl_3(q)\times \spl_3(q))\rtimes2$ when $q\equiv 0,2$ mod $3$ by~\cite[Section 4.8.9]{wilson_2009}, none of which are power-chordal because the direct product of two non-abelian simple groups is not power-chordal via Lemma~\ref{lem:direct_prod}. 
\end{proof}

\begin{lem}
	The groups $E_6(q)$, 
	$\negphantom{x}\phantom{x}^2E_6(q)$, $E_7(q)$ and $E_8(q)$ are not power-chordal.
\end{lem}
\begin{proof}
	All these groups contain $F_4(q)$ by Lemma~\ref{spl3_levi_factor}, and these groups are not power-chordal by Lemma~\ref{lem:f4q}.
\end{proof}

\subsection{Sporadic groups}
The last groups to consider are the $26$ sporadic groups. The result below shows that none of the sporadic groups are power-chordal and we prove this by either finding (maximal) subgroups of certain sporadic groups that are non-chordal or finding commuting elements of coprime order that satisfy Lemma~\ref{lem:conjug_cyclic_property_sufficient}. 

\begin{lem} \label{lem:sporadics}
	Let $G$ be a sporadic simple group. Then $G$ is not power-chordal. 
\end{lem}
\begin{proof}
	The power graph of the Mathieu group $\mathrm{M}_{11}$ is not chordal, via calculations we performed in GAP~\cite{GAP4} and SageMath~\cite{sagemath}. Recall that the Mathieu group $\mathrm{M}_{11}$ is a subgroup of all the other sporadic simple groups except seven of them: $\mathrm{J}_1$, $\mathrm{M}_{22}$, $\mathrm{J}_2$, $\mathrm{J}_3$, $\mathrm{He}$, $\mathrm{Ru}$ and $\mathrm{Th}$ (see~\cite{atlas}). Except for the Janko group $\mathrm{J}_1$, we can rule out the remaining groups by finding maximal subgroups whose power graphs are not chordal.
The Janko group $\mathrm{J}_2$ contains $A_4 \times A_5$, which is not power-chordal.
The Rudvalis group $\mathrm{Ru}$ contains $A_8$, which is not power-chordal.
The Janko group $\mathrm{J}_3$ contains $((C_2\times Q_8)\rtimes C_2)\rtimes A_5$ which is not power-chordal. The Held group $\mathrm{He}$ contains $\psl_3(2)\times S_4$, which contains elements of order $42$ and is thus not power-chordal. The Thompson group $\mathrm{Th}$ contains $\psl_3(3)$, which is not power-chordal.
The standard representation of the Mathieu group $\mathrm{M}_{22}$ in GAP~\cite{GAP4} contains elements 
\[
	x=(1,4,16)(2,15,12)(3,8,18)(5,13,9)(6,11,14)(7,22,17)
\]
of order $3$ and
\[
	y=(1,14)(2,9)(4,6)(5,15)(10,21)(11,16)(12,13)(19,20)
\]
of order $2$. We find that $xy=yx$, $\langle y \rangle \not \trianglelefteq C_{\mathrm{M}_{22}}(x)$, and $\langle x \rangle\not \trianglelefteq C_{M_{22}}(y)$ hold. By Lemma~\ref{lem:conjug_cyclic_property_sufficient}, $\pow(\mathrm{M}_{22})$ is not chordal. Finally, we are left with $\mathrm{J}_1$, all of whose maximal subgroups are power-chordal. Let $x$ and $y$ in $G$ be of orders $3$ and $2$, respectively, such that $xy=yx$. Then it holds $C_G(x)=\langle x \rangle \times D_{10}$ and $C_G(y)=C_2\times A_5$. Then  $x$ and $y$ satisfy Lemma~\ref{lem:conjug_cyclic_property_sufficient}, and the result follows.
\end{proof}

This completes our analysis of power graphs of finite simple groups with respect to chordality. 
In total, the present section constitutes a proof of Theorem~\ref{thm:main}. 

\section{Beyond simple groups: interaction between cycles and paths}
In this final section we develop a relation between paths and cycles in the power graph. More precisely, we give
general bounds on the length of the longest induced path in a chordal power graph, depending only on the (number of) prime divisors of the group. Here, the length of an induced path in $\pow(G)$ always refers to the number of its vertices.

Recall that an induced path in the power graph has alternating arc directions in the directed power graph (Lemma~\ref{lem:paths_alternating}).

\begin{lem}\label{lem:in_vertex_prime_power}
	Let $G$ be a group such that $|G|$ is not divisible by $p^2q^2$ or $pqr$ for distinct primes $p$, $q$ and $r$.
	Let $\wp$ be an induced path in $\pow(G)$ with at least four vertices. Then
	the orders of consecutive in-vertices
	of $\wp$ are prime-powers over distinct primes.
\end{lem}
\begin{proof}
	Since in- and out-vertices alternate in $\wp$ and $\wp$ has length at least four, each in-vertex in $\wp$ is joined with some out-vertex $g\in\wp^+$ that has degree two in $\wp$. The order of $g$ cannot have more than two prime divisors and it is not divisible by
	$p^2q^2$ for primes $p\neq q$ by assumption. By definition, the in-vertices joined with $g$, say $x$ and $y$, are contained in $\langle g\rangle$ but they are not adjacent with each other. Thus the orders of $x$ and $y$ must be incomparable and $\mathrm{lcm}(|x|,|y|)$ divides
	$pq^m$ for primes $p\neq q$ and some $m\in\mathbb{N}$, only leaving the possibility that $|x|$ and $|y|$ are (powers of) $q$ and $p$, respectively. Finally, again noting that $x$ and $y$ both lie in $\langle g\rangle$, we see that $|x|$ and $|y|$ must be powers of distinct primes. 
\end{proof}

\begin{lem}\label{lem:power_reduced}
	Let $G$ be a group such that $|G|$ is not divisible by $p^2q^2$ or $pqr$ for distinct primes $p$, $q$ and $r$.
	If $\wp$ is an induced path in $\pow(G)$ then there is a power-reduced path in $\pow(G)$
	of the same length as $\wp$, or there is a cycle in $\pow(G)$ whose length is at least $4$ but at most the length of $\wp$.
\end{lem}
\begin{proof}
	By Lemma~\ref{lem:in_vertex_prime_power}, each $g\in\wp^-$ has prime-power order.
	Starting from $\wp$, we successively replace in-vertices by powers of prime order and argue that the resulting graph is still a path of length $|\wp|$. Let $\wp'$ be the graph obtained by $\wp$ after replacing one in-vertex $x\in\wp^-$ of order $q^m$ with $x':=x^{q^{m-1}}$.
	If $\wp'$ is not a path, then there must be new edges incident with $x'$. If there is no induced cycle of length at least $4$ in $\wp'$, then $x'$ must be joined 
	with another in-vertex $y$ such that $x$ and $y$ have a common neighbor $g\in\wp^+$. But according to Lemma~\ref{lem:in_vertex_prime_power}, 
	$x$ and $y$ have coprime orders, a contradiction. 
\end{proof}

In particular, we might assume that paths in the power graph of power-chordal groups come with the restricted structure described in the previous lemma. Recall that the power graph of a group $G$ is invariant under the action of $\aut(G)$. The idea is to consider long paths in $\pow(G)$ and shift them around via suitable automorphisms to either form cycles or otherwise obtain structural restrictions on the group itself.
\begin{lem}\label{lem:bound_path_len_by_max_len}

	Let $G$ be a group such that $|G|$ is not divisible by $p^2q^2$ or $pqr$ for distinct primes $p$, $q$ and $r$.
	Let $\wp=(\wp_1,\dots,\wp_{\ell})$ be a power-reduced path in $\pow(G)$ with $\wp_1\in\wp^-$ and assume
	that $\pow(G)$ does not contain cycles whose length is in $\{4,\dots,2\ell-6\}$. 
	Let $L$ be the maximal length of any path in $G$. Then $\wp_1$ normalizes each of the cyclic groups $\langle\wp_i\rangle$ with
	$i< l-\frac{L-3}{2}$.
\end{lem}
\begin{proof}
 We note that $\wp_1$ always normalizes $\langle x\rangle=\langle\wp_1\rangle$, $\langle\wp_2\rangle$ and $\langle\wp_3\rangle$ (by definition of the power graph,
	$x$ even centralizes these groups). If $x:=\wp_1$ does normalize
	$\langle \wp_i\rangle$ for $i<j$ but does not normalize $\langle \wp_j\rangle$, then the paths $(\wp_{j},\dots,\wp_{\ell})$ and $(\wp^x_{j},\dots,\wp^x_{\ell})$ either have no edges between them in $\pow(G)$, and then they can be combined to a path of length $2(\ell-j+1)+1$ via $\wp_{j-1}$, or they do have edges between them, and then, since we assume $\langle \wp_j\rangle\neq\langle \wp_j^x\rangle$, there must be an induced cycle on
	$\{\wp_i^y\mid j-1\leq i\leq \ell, y\in\{x,1\}\}$. The length of this cycle would be in $\{4,\dots,2\ell-6\}$, contradicting the assumptions of the lemma.
\end{proof}

This further restricts the structure of long paths in chordal power graphs, since for distinct primes $p$ and $q$ we have that a copy of $C_p$ can only normalize a copy of $C_q$ if $p$ divides $q-1$ or if it actually centralizes the copy of $C_q$. But the latter case only occurs among consecutive in-vertices, as the following lemma shows. 

\begin{lem}\label{lem:path_start_no_centralizer}
	Let $\wp$ be a power-reduced path of length $\ell$ in $\pow(G)$ and assume that $\wp$ is not contained in any induced cycle of $\pow(G)$. Then
	each $x\in\wp^-$ only centralizes elements in $\wp^-$ of order coprime to $|x|$ if they have distance $2$ from $x$.
\end{lem}
\begin{proof}
	Consider $x,y\in\wp^-$ such that $|x|$ and $|y|$ are coprime, centralize each other, and have distance greater than $2$ in
	$\wp$. Consider the vertex $xy$ in $\pow(G)$. If $xy$ is joined with some path-vertex $\wp_i\in\wp\setminus\{x,y\}$, then 
	we have $\wp_i\in\wp^+$ and $\wp_i$ is also joined with $x$ and $y$ by definition, contradicting the distance assumption.
	But otherwise $\wp\cup\{xy\}$ forms an induced cycle, contradicting the assumption on cycles in $\pow(G)$.
\end{proof}

We can now bound the length of a longest path in a chordal power graph in terms of the number of prime divisors of the group order.
\begin{thm}\label{thm:bound_max_len}
	Let $G$ be a finite power-chordal group. If $\wp$ is an induced path on $L$ vertices in $\pow(G)$, then $L\leq 19$. In particular, if $G$ is any finite group, then $\pow(G)$ is chordal if and only if it neither contains induced cycles of length at most $20$ nor contains an induced path of length $20$.
\end{thm}
\begin{proof}
	Let $L$ be the length of a longest path in $\pow(G)$. Assume $L$ is at least $20$ and let $\wp$ be a path of maximal length in $\pow(G)$.
	By Lemma~\ref{lem:power_reduced}, we may assume that $\wp$ is power-reduced. By omitting at most two vertices from $\wp$, we may further assume that $\wp$ starts with an element $x\in\wp^-$ and also ends with an element in $\wp^-$. Denote the length of the resulting path (we simply call $\wp$ again) by $\ell$ and denote the $i$-th vertex of $\wp$ by $\wp_i$ (starting from one of the end-points of $\wp$).
	Set $n:=\lceil \frac{\ell}{2}\rceil$, so $n\geq \frac{L-1}{2}$ and since $\wp$ starts and ends in in-vertices, it holds that $\wp_n$ is
	again an in-vertex. By Lemma~\ref{lem:bound_path_len_by_max_len}, $x$ normalizes all cyclic subgroups generated by elements in the sub-path $(\wp_1,\dots,\wp_n)$. By (almost) maximality of $\wp$ it follows that $\langle \wp_3\rangle$ is normal in $C_G(x)$, otherwise
	there would be another cyclic group $\langle y\rangle$ of order $|\wp_3|$ centralizing $x$, but $|\wp_3|$ is a prime distinct from $|x|$, and thus $\wp$ could be extended beyond $x$ by two more vertices ($y$ cannot be contained in $\wp$ by Lemma~\ref{lem:path_start_no_centralizer}), which is impossible by the choice of $\wp$. Now assume that some element
	$\wp_i$ with $2\leq i\leq n$ has order $|x|$. Note that then $i\geq 5$ by the structure of $\wp$. Then $x$ normalizes $\langle \wp_i\rangle$, thus it centralizes $\wp_i$ and in turn, $\wp_i$ normalizes the unique cyclic group of order $|\wp_3|$ in $C_G(x)$. By Lemma~\ref{lem:bound_path_len_by_max_len}, $\wp_3$ also normalizes $\langle \wp_i\rangle$, which is only possible if $\wp_3$ and $\wp_i$ centralize each other. By Lemma~\ref{lem:path_start_no_centralizer}, $i=5$ follows. In particular, since $n>7$, the in-vertex $\wp_n$ does not have order $|x|$. But then $\wp_1$ and $\wp_5$ both induce fixed point free automorphisms of the same order on $\langle \wp_n\rangle$, which has a cyclic automorphism group. Thus, some $g\in\langle\wp_1,\wp_5\rangle$ centralizes $\wp_n$ and by the structure of $\wp$, $g$ also centralizes $\wp_3$. We can thus replace $\wp_1$ with $g$ and obtain a cycle in $\pow(G)$, induced on $(g,\wp_2,\dots,\wp_n)$, contradicting the assumption that $G$ is power-chordal.
	So we proved that in fact $x=\wp_1$ is the only vertex $\wp_i$ with $i\leq n$ of order $|x|$. Then, since orders of consecutive in-vertices are distinct, the primes $|\wp_1|$, $|\wp_3|$, and $|\wp_5|$ are pairwise distinct. If $|\wp_n|\neq |\wp_5|$, then
	the non-trivial semi-direct product $\langle \wp_5\rangle\rtimes\langle \wp_1\rangle$ acts without fixed points on the cyclic group
	generated by $\wp_n$, a contradiction. It follows that $|\wp_5|=|\wp_n|$ holds. We can repeat the exact same argument with
	$\wp_1$ and $\wp_7$, unless $\wp_7$ centralizes $\wp_n$. But then Lemma~\ref{lem:path_start_no_centralizer} implies $n\in\{7,9\}$.
	In conclusion, $\ell\leq 17$ and $L\leq 19$.
\end{proof}

We finish the paper by noting three consequences of the discussion above. First, we can strengthen the statement of the previous theorem for more restrictive group classes.

\begin{cor}
	Let $G$ be a power-chordal group of order $p^aq^b$ with distinct primes $p$ and $q$ and $a,b\in\mathbb{N}$.
	If $\wp$ is a path in $\pow(G)$ of length $L$, then $L\leq 15$.
\end{cor}
\begin{proof}
	In the proof of Theorem~\ref{thm:bound_max_len}, we deduced that the order of $\wp_1$ is unique among the in-vertices of $\wp$, provided that $n>7$ holds. But the remaining in-vertices have prime order and consecutive in-vertices cannot have the same order.
	Thus, assuming that $|G|$ has only two distinct prime divisors, we conclude that $L$ is at most $15$.  
\end{proof}

The next observation follows from the proof of Lemma~\ref{lem:bound_path_len_by_max_len}.
\begin{cor}
	Let $G$ be power-chordal and let $\wp$ be a power-reduced path in $G$. If $\wp$ has maximal length among all paths of $\pow(G)$, then the group generated by $\wp$ has a cyclic normal subgroup.
\end{cor}

Finally, note that the special structure of power-reduced paths establishes a connection between paths
in the power graph and special paths in the commuting graph. Indeed, if $\wp$ is a power-reduced path in $\pow(G)$ then the induced subgraph of $\mathrm{Com}(G)$ induced on $\wp^-$
	is a path in which consecutive vertices have distinct prime orders.
\begin{cor}
	If $G$ is power-chordal and $\mathrm{Com}(G)$ contains an induced path of length $L$ in which consecutive vertices have distinct prime orders, then $L\leq 9$.
\end{cor}

\section{Acknowledgments}

The research leading to these results has received funding
from the European Research Council (ERC) under the European
Union’s Horizon 2020 research and innovation program
(EngageS: grant agreement No. 820148) and from the German Research Foundation DFG (SFB-TRR 195 ``Symbolic Tools in Mathematics and their Application''). 

We thank Luke Morgan for helpful discussions regarding Lemma~\ref{spl3_levi_factor}.

\section{Statements and declarations}
We declare that there are no competing interests or conflicts of interests.
\section{Data availability}
Our manuscript has no associated data.

\bibliographystyle{plain}
\bibliography{refs}

\end{document}